\documentclass[11pt,reqno]{amsart}
\usepackage{latexsym,amssymb,amsfonts,amsmath,amsthm}
\usepackage{tocvsec2}
\usepackage[noadjust]{cite}
\usepackage{hyperref}
\usepackage[margin=1in]{geometry}

\newtheorem{theorem}{Theorem}[section]
\newtheorem{proposition}[theorem]{Proposition}

\theoremstyle{definition}
\newtheorem{definition}[theorem]{Definition}
\newtheorem{example}[theorem]{Example}
\newtheorem{remark}[theorem]{Remark}
\newtheorem{conj}[theorem]{Conjecture}

\numberwithin{equation}{section}
\numberwithin{table}{section}

\newcommand{\ch}{\operatorname{char}}
\newcommand{\chG}{\underline{\operatorname{ch}}_G}
\newcommand{\qpoch}[3]{(#1;#2)_{{#3}}}
\renewcommand{\geq}{\geqslant}
\renewcommand{\leq}{\leqslant}

\allowdisplaybreaks

\begin{document}

\title[Log-concavity of parabolic Verma characters and restricted Kostant
partitions]{Log-concavity of characters of parabolic Verma
modules,\\ and of restricted Kostant partition functions}

%\date{\today}

\author{Apoorva Khare}
\address[A.~Khare]{Department of Mathematics, Indian Institute of
Science, Bangalore -- 560012, India; and Analysis and Probability
Research Group, Bangalore -- 560012, India}
\email{\tt khare@iisc.ac.in}

\author{Jacob P.\ Matherne}
\address[J.P.~Matherne]{Department of Mathematics, North Carolina State
University, Raleigh, NC 27695, USA}
\email{\tt jpmather@ncsu.edu}

\author{Avery St.~Dizier}
\address[A.~St.~Dizier]{Department of Mathematics, Michigan State
University, East Lansing, MI 48824, USA}
\email{\tt stdizier@msu.edu}


 %$n$-dimensional polytopes}

\begin{abstract}
In 2022, Huh--Matherne--M\'esz\'aros--St.~Dizier 
showed that normalized Schur polynomials are Lorentzian, thereby yielding
their continuous (resp.\ discrete) log-concavity on the positive orthant
(resp.\ on their support, in type $A$ root directions).  A
reinterpretation of this result is that the characters of
finite-dimensional simple representations of
$\mathfrak{sl}_{n+1}(\mathbb{C})$ are denormalized Lorentzian
(DL).  In the same paper, these authors also showed that
shifted characters of Verma modules over
$\mathfrak{sl}_{n+1}(\mathbb{C})$ are DL.

In this work we extend these results to a larger family of modules that
subsumes both of the above: we show that shifted characters of all
parabolic Verma modules over $\mathfrak{sl}_{n+1}(\mathbb{C})$ are
denormalized Lorentzian. The proof involves certain graphs on $[n+1]$;
more strongly, we explain why the character (i.e., generating function)
of the Kostant partition function of any loopless multigraph on $[n+1]$
is Lorentzian after shifting and normalizing.

We then show that parabolic Vermas form a ``maximal'' class
with log-concave (hence DL) characters. Namely, log-concavity fails in
greater generality along three natural directions:
(1)~it does not hold for every simple Lie type,
(2)~nor for a larger universal family of highest weight
modules, the higher order Verma modules, even in type $A$,
and
(3)~it does not always hold for important generalizations of Schur
polynomials: the Jack and Macdonald polynomials.

Finally, we extend these results to parabolic (i.e.\ ``first
order'') and higher order Verma modules over the semisimple Lie algebras
$\oplus_{t=1}^T \mathfrak{sl}_{n_t+1}(\mathbb{C})$.
We also partially resolve a conjecture of Huh et al.\ on the
DL property for integral highest weight simple modules.
\end{abstract}

\maketitle

\settocdepth{section}
\tableofcontents

%{{{1 Section 1 - Introduction and main results
\section{Introduction and main results}

This paper adds to the classical and recent works that study symmetric
functions (in finitely many variables) from an analysis perspective,
specifically, their behavior when the variables are evaluated on the
positive orthant. This includes the 2011 paper of
Cuttler--Greene--Skandera~\cite{CGS} (which includes a literature survey
with links to numerous classical works, by Maclaurin, Newton, Muirhead,
Schur, Gantmacher, and others),
as well as subsequent works by Sra \cite{Sra},
McSwiggen--Novak \cite{MN},
one of us with Tao \cite{KT},
and by the other two of us with Huh and M\'esz\'aros \cite{HMMS}.
In particular, this last work contained the following two results
\cite[Theorem~3 and Proposition~11]{HMMS}:
\begin{enumerate}
\item Normalized Schur polynomials are Lorentzian (see~\eqref{Enormal}
below for the definition of ``normalized''). This implies their
``continuous'' log-concavity on the positive orthant, as well as the
discrete log-concavity of their coefficients (the Kostka numbers) along
type $A$ root directions -- i.e.\ for $\mathfrak{sl}_{n+1}(\mathbb{C})$.

\item The Kostant partition function, i.e.\ the character of any Verma
module (which encodes its weight multiplicities) over
$\mathfrak{sl}_{n+1}(\mathbb{C})$, is also discretely log-concave along
type $A$ root directions. More strongly, and as in~(1), its normalization
is Lorentzian \cite[Proposition~13]{HMMS}, hence continuously
log-concave.
\end{enumerate}

Note that Schur polynomials are the characters of finite-dimensional
simple modules over $\mathfrak{sl}_{n+1}(\mathbb{C})$. It is natural to
ask if there is a class of representations which subsumes (or
interpolates between) these modules and Vermas, and such that the above
log-concavity results (both continuous and discrete) can be proved for
all modules in this larger class.

The goal of this paper is to provide an affirmative answer to these
questions, via \textit{parabolic Verma modules} $M(\lambda, J)$.  These
are indexed by a highest weight $\lambda$ and a subset $J$ of simple
roots/simple reflections -- equivalently, by $\lambda$ and a
parabolic subgroup $W_J$ of the Weyl group $W = S_{n+1}$ of
$\mathfrak{sl}_{n+1}(\mathbb{C})$. (See Section~\ref{Spvm} for notation
and details on parabolic Verma modules.) We formalize this via our first
main result, Theorem~\ref{T1} below. First, we set some notation for
the entire paper.

\begin{definition}
	Throughout, $\mathbb{N}$ denotes the nonnegative integers, and
	$[n+1] := \{ 1, \dots, n+1 \}$ for $n \in \mathbb{N}$. By a
	monomial in $x = (x_1, \dots, x_m)$ we mean $x^\mu :=
	\prod_{j=1}^m x_j^{\mu_j}$, where all $\mu_j \in \mathbb{Z}$. For
	$\mu \in \mathbb{N}^m$ define $\mu! := \prod_{j=1}^m \mu_j!$; now
	define the {\em normalization operator} on the space of Laurent
	series/generating functions over a field $\mathbb{F}$ of
	characteristic zero, via restriction to the monomials of
	nonnegative degree in each variable:
	\begin{equation}\label{Enormal}
	N \left( \sum_{\mu \in \mathbb{Z}^m} c_\mu x^\mu \right) :=
	\sum_{\mu \in \mathbb{N}^m} c_\mu \frac{x^\mu}{\mu!}.
	\end{equation}
	Finally, we write $\varepsilon_1, \dots, \varepsilon_{n+1}$ for
	the coordinate basis of $\mathbb{F}^{n+1}$ (or
	$\mathbb{Z}^{n+1}$) for $n \in \mathbb{N}$.
\end{definition}

We next recall the two notions of log-concavity that are
discussed in this work.

\begin{definition}
A polynomial $h(x) = \sum_{\mu \in \mathbb{N}^m} c_\mu x^\mu$ in the variables $x_1,
\dots, x_m$ is \textit{continuously log-concave} if either $h \equiv 0$
or $h>0$ on the positive orthant $\mathbb{R}^m_{>0}$ and $\log(h)$ is
concave here.
If $h$ is homogeneous, it is said to be \textit{discretely
log-concave}, or to have \textit{discretely log-concave coefficients (in
type $A$ root directions)}, if
\[
c_{\mu}^2 \geq c_{\mu+\varepsilon_i-\varepsilon_j}
c_{\mu-\varepsilon_i+\varepsilon_j} \quad\mbox{for every
$\mu\in\mathbb{N}^m$ and $i,j\in[m]$.}
\]
\end{definition}

(Log-)Concavity is a well-studied notion, while its
discrete univariate version has also been investigated since Newton's
inequalities and total positivity. The multivariate version is less
studied; see Section~\ref{Sadlc} for some recent positive (and two novel
negative) results.

\subsection{Lorentzian polynomials}

Lorentzian polynomials, introduced in the groundbreaking work of
Br\"and\'en and Huh \cite{BH} (and independently in
\cite{AGV21,ALOGV24a,ALOGV24b} under the name \emph{completely
log-concave polynomials}), provide a powerful unifying framework
connecting discrete and continuous log-concavity. Lorentzian polynomials
have since seen myriad applications across mathematics 
\cite{BH,EH20,MS21,HMMS,BL23,BLP23,Ros23,ALOGV24b,BES24,HMV24,MMS24}.

\begin{definition}[{\cite[pp.~822--823]{BH}}]
A homogeneous polynomial $h\in\mathbb{R}[x_1,\ldots,x_m]$ of degree $d$
is called \emph{Lorentzian} if the following conditions hold:
\begin{enumerate}
	\item The coefficients of $h$ are nonnegative;
	\item The support of $h$ is M-convex.\footnote{A subset
	$J$ of $\mathbb{N}^m$ is \textit{M-convex} if for $\alpha \neq
	\beta \in J$ and any $i \in [m]$ with $\alpha_i > \beta_i$, there
	is an index $j$ with $\alpha_j < \beta_j$ and $\alpha -
	\varepsilon_i + \varepsilon_j \in J$.}
	\item For any $i_1,\ldots,i_{d-2}\in [m]$, the quadratic form
	$\frac{\partial}{\partial x_{i_1}}\frac{\partial}{\partial
	x_{i_2}}\cdots \frac{\partial}{\partial x_{i_{d-2}} }h$ has at
	most one positive eigenvalue.
\end{enumerate}		
We say $h$ is \emph{denormalized Lorentzian} if $N(h)$
(see~\eqref{Enormal}) is Lorentzian.
\end{definition}

We now collect together the key properties of Lorentzian
polynomials that are used below.

\begin{theorem}\label{thm:BHdoubleLC}
Suppose $h(x) = \sum_{\mu\in\mathbb{N}^m} c_\mu x^{\mu}$ is denormalized
Lorentzian and nonzero. Then:
\begin{enumerate}
\item \cite[Theorem~2.30]{BH}
$N(h)$ is continuously log-concave.\footnote{In fact, the continuous
log-concavity of all derivatives of $N(h)$ was introduced by Gurvits
\cite{Gurvits} under the name \textit{strong log-concavity}, and in
\textit{loc.\ cit.}\ Br\"and\'en--Huh showed that this is equivalent to
the Lorentzianity of $N(h)$.}

\item \cite[Proposition~4.4]{BH}
$h$ is discretely log-concave.

\item \cite[Corollary 3.8]{BH}
If moreover $g(x)$ is also denormalized Lorentzian, then so is $gh$.
\end{enumerate}
\end{theorem}

\subsection{Main results}

\begin{theorem}\label{T1}
For any integer $n>0$ and parabolic Verma module $M(\lambda,J)$ over
$\mathfrak{sl}_{n+1}(\mathbb{C})$, and all $\delta \in \mathbb{N}^{n+1}$,
the normalization $N(x^\delta \cdot \ch M(\lambda, J))$ is Lorentzian.
Consequently, one has both a continuous and discrete version of
log-concavity:
\begin{enumerate}
\item $N(x^\delta \cdot \ch M(\lambda, J))$ is either identically zero or
log-concave as a function on the positive orthant
$\mathbb{R}_{>0}^{n+1}$, and 
\item if $\mu(ij) := \mu +
\varepsilon_i - \varepsilon_j$ for $i,j \in [n+1]$, then
\begin{equation}\label{EdiscLC}
(\dim M(\lambda,J)_\mu)^2 \geq \dim M(\lambda,J)_{\mu(ij)} \cdot
\dim M(\lambda,J)_{\mu(ji)}, \quad \forall \mu \in \mathfrak{h}^*, \ i,j
\in [n+1].
\end{equation}
\end{enumerate}
\end{theorem}

This result specializes to \cite[Theorems~1--3]{HMMS} for
finite-dimensional simple modules/Schur polynomials, by setting $J=I$
and $\delta = 0$. Similarly, one recovers \cite[Propositions~11,
13]{HMMS} for Verma modules/the (usual) Kostant partition function, by
setting $J = \emptyset$. 

Here is a second theme that emerged during the course of proving
Theorem~\ref{T1}: we were naturally led to exploring connections between
parabolic Verma characters, the associated restricted Kostant partition
functions, and the theory of flow polytopes. 
In the flow polytope language, the novel ingredient in the proof
of~\eqref{EdiscLC} involves working with flow polytopes of directed
simple graphs with vertex set $[n+1]$ whose omitted edges comprise an
order-ideal in the root poset.
The following result shows more strongly that the restricted Kostant
partition function for an \textit{arbitrary} set of edges is discretely
log-concave -- and continuously so as well.

\begin{theorem}\label{T2}
Let $G$ be any loopless directed finite multigraph on $[n+1]$ with edges
directed $i \to j$ for $i<j$. Then for any $v\in \mathbb{Z}^{n+1}$ and
$i,j\in [n+1]$,
\[
K_G(v)^2 \geq 
K_G(v + \varepsilon_i - \varepsilon_j) 
K_G(v + \varepsilon_j - \varepsilon_i),
\]
where $K_G(\cdot)$ denotes the restricted Kostant partition function
of~$G$ (see Definition~\ref{Dkpf}).
More strongly, if $\chG$ denotes the generating function of
$K_G$, then $N(x^\delta \cdot \chG(x))$ is Lorentzian for all $\delta \in
\mathbb{N}^{n+1}$.
\end{theorem}

Note the discrete log-concavity assertion of Theorem \ref{T2} is
also proved in \cite[Corollary 5.2]{MS21} using Lorentzian projections of
the integer-point transforms of flow polytopes.

Our next result shows that the (discrete) log-concavity of parabolic
Verma modules $M(\lambda,J)$ is a ``tight'' improvement over the
results in~\cite{HMMS} for Vermas and finite-dimensional simples, from
the viewpoint of representation theory. The family of parabolic Verma
modules was shown in recent work~\cite{Teja} to be a part of the
\textit{higher order Verma modules}, which enjoy similar universal
properties to $M(\lambda,J)$. In this language, usual Verma modules are
of zeroth order, while parabolic Vermas are of first order -- and by
Theorem~\ref{T1}, all of their characters are log-concave.

\begin{theorem}\label{T4}
Let $m \geq 2$ and consider any $m$th order Verma
$\mathfrak{sl}_{n+1}(\mathbb{C})$-module $V$ that lacks singleton holes.
Then $\ch V$ is not (discretely) log-concave.
\end{theorem}

Thus, parabolic Verma modules are the ``best possible'' among these
universal highest weight modules as far as log-concavity of their
character goes.

Our final result extends the results above -- and hence some of the
results in \cite{HMMS} -- from (parabolic) Verma modules over
$\mathfrak{sl}_{n+1}(\mathbb{C})$ to those over a larger family of
complex semisimple Lie algebras:

\begin{theorem}\label{T5}
Let $n_1, \dots, n_T$ be positive integers, and let $\mathfrak{g} =
\oplus_{t=1}^T \mathfrak{sl}_{n_t+1}(\mathbb{C})$ with positive roots
$\Delta$. Then for all $\delta \in \mathbb{N}^{d}$, where $d =
\sum_{t=1}^T (n_t+1)$, the normalized shifted character $N(x^\delta \cdot
\ch M(\lambda,J))$ of every parabolic Verma $\mathfrak{g}$-module is
Lorentzian -- and in particular, continuously and discretely (along all
root directions in $\Delta$) log-concave, as in Theorem~\ref{T1}.
\end{theorem}

In particular, the (normalized shifted) character of every Verma module
and every finite-dimensional simple module over such $\mathfrak{g}$ is
continuously and discretely log-concave. Moreover, in Theorem~\ref{Thvm}
we characterize when higher order Verma modules over these $\mathfrak{g}$
fail to have discretely log-concave characters.

We end by extending \cite[Conjecture~12]{HMMS} in two ways:
(i)~to all weights $\lambda$ (possibly non-integral); and
(ii)~to the setting of Theorem~\ref{T5}:

\begin{conj}\label{Conj}
Let $\mathfrak{g} = \oplus_{t=1}^T \mathfrak{sl}_{n_t+1}(\mathbb{C})$ as
above. For an arbitrary weight $\lambda \in
\mathfrak{h}^*$, let $M(\lambda) \twoheadrightarrow V(\lambda)$ be the
unique simple highest weight module with highest weight $\lambda$. Then
for all $\delta \in \mathbb{N}^{d}$, where $d = \sum_{t=1}^T (n_t+1)$,
the normalized shifted character $N(x^\delta \cdot \ch V(\lambda))$ is
Lorentzian.
\end{conj}

It is clear that by setting $T=1$ in Conjecture~\ref{Conj}, and
restricting to integral highest weights $\lambda$, one recovers
\cite[Conjecture~12]{HMMS}. As we explain in Appendix~\ref{Sconj},
(a)~one can partly go the other way: if \cite[Conjecture~12]{HMMS} holds
then one can prove Conjecture~\ref{Conj} -- only for \textit{integral}
highest weights. Moreover,
(b)~Theorems~\ref{T1} and~\ref{T5} make partial (positive) progress in
verifying these conjectures, via Jantzen's simplicity criterion
\cite[Satz~4]{Jantzen}, which classifies the pairs $(\lambda, J)$ for
which $M(\lambda, J)$ is simple -- and hence its $x^\delta$-shifted
character is denormalized Lorentzian by Theorem~\ref{T5}. For instance,
for generic $\lambda$ in the space $\mathfrak{h}^*$ of (highest) weights
that avoid a countable collection of hyperplanes $H_{\alpha,n}$ (indexed
by a root $\alpha$ and an integer $n$), \cite[Conjecture~12]{HMMS}
follows from results in \cite{HMMS} and a fact on Verma modules. However,
these do not suffice to show the conjecture for the generic $\lambda$
among the remaining weights, i.e.\ the $\lambda$ that lie in exactly one
of these hyperplanes $H_{\alpha,n}$. We show that if $\alpha$ is simple
then our results confirm both conjectures -- for not necessarily integral
weights~$\lambda$.

\subsection*{Organization}

In Section~\ref{Spvm}, we introduce parabolic Verma modules and provide
background results for them.
In Section~\ref{Slorentzian}, we explain how parabolic Vermas connect to
restricted Kostant partition functions, and then show Theorems~\ref{T1}
and~\ref{T2}.
Next, in Section~\ref{Sflow} we recall the Lidskii volume formula and the
Alexandrov--Fenchel inequality, and use them to give an alternative proof of the
discrete log-concavity in Theorem~\ref{T2} -- but not in
Theorem~\ref{T1}, since discrete log-concavity fails to be preserved
under products (we provide counterexamples).

Section~\ref{Shigherorder} discusses the failure of discrete and hence
continuous log-concavity for normalized characters -- both for
higher-order Verma modules over $\mathfrak{sl}_{n+1}(\mathbb{C})$ (i.e.,
we show Theorem~\ref{T4}), as well as
for Verma modules outside type $A$. (Recall from~\cite[Figure 2]{HMMS}
that not all finite-dimensional simple modules in type $C_2$ have
log-concave characters.)
Moreover, we will also record here that certain important one- or
two-parameter generalizations of Schur polynomials -- namely,
Hall--Littlewood, Jack, and Macdonald polynomials -- do not possess
log-concave sequences of coefficients for all admissible values of the
parameters. Thus, parabolic (i.e.\ first-order) Verma modules of type $A$ are a ``maximal'' family of ``universal'' highest weight
modules that possess log-concave characters.
However, this is not true for higher order Verma modules, nor for Vermas or finite-dimensional simples across all Lie types,
and also not for Jack or Macdonald generalizations of the
characters.
Thus our main Theorem~\ref{T1} is ``tight'' along multiple fronts.

Next, in Section~\ref{Sfinal} we work over a direct sum of
$\mathfrak{sl}_{n+1}$'s and show Theorem~\ref{T5} (or its more precise
formulation in Theorem~\ref{Thvm}). Finally, we explain in detail in
Appendix~\ref{Sconj}, how Theorems~\ref{T1} and~\ref{T5} partly resolve
Conjecture~\ref{Conj} (hence \cite[Conjecture~12]{HMMS}).
%}}}

%{{{1 Section 2 - Background on parabolic Verma modules
\section{Background on parabolic Verma modules}\label{Spvm}

As the above account suggests, the results and proofs in this work
involve tools and ideas from several different subfields:
(a)~representation theory of Lie algebras;
(b)~flow polytopes and vector partition functions (from algebraic
combinatorics); and
(c)~log-concave/Lorentzian polynomials (in combinatorics/analysis).
Thus, a secondary goal of this work is to provide brief introductions to
these topics, as well as relatively detailed proofs, in the interest of
making this work as self-contained as possible for the readers from
various backgrounds/communities who might not be well-versed with a
subset of these topics. The familiar reader should feel free to skim
through (or even skip) these accounts, while taking with them the
notation that is set below.

\subsection{Notation for semisimple Lie algebras}

This subsection and the next two discuss semisimple Lie algebras -- e.g.\
$\mathfrak{sl}_{n+1}(\mathbb{C})$ -- and parabolic Verma modules over
them. This includes explaining why these are the ``natural'' class of
modules that unify/subsume both Verma modules and finite-dimensional
simple modules. See \cite{Hum} for a more detailed account of these
topics.

Let $\mathfrak{g}$ be any complex semisimple Lie algebra (for our
results, we work with $\mathfrak{g} = \mathfrak{sl}_{n+1}(\mathbb{C})$
with $n > 0$).
Let $\mathfrak{h}$ denote the Cartan subalgebra (correspondingly for us,
the space of traceless diagonal matrices), and fix a base of simple roots
$\{ \alpha_i : i \in I \}$ in $\mathfrak{h}^*$ (for us, $I = [n]$ and the
simple root $\alpha_i := \varepsilon_i - \varepsilon_{i+1}$ sends a
diagonal matrix $h$ to the difference $h_{ii} - h_{i+1,i+1}$ of diagonal
entries, for $i \in I$).
Then $\mathfrak{g}$ is generated as a Lie algebra by Chevalley
generators:
\begin{itemize}
\item the simple raising operators $e_i, i \in I$ (for us, the elementary
matrices $E_{i,i+1}$),
\item the simple lowering operators $f_i, i \in I$ (for us, the elementary
matrices $E_{i+1,i}$),
\item and their commutators $h_i = [e_i, f_i] \in \mathfrak{h}$ (for us,
the diagonal matrices $E_{ii} - E_{i+1,i+1}$).
The elements $h_i, i \in I$ form a basis of $\mathfrak{h}$, and
correspondingly, the simple roots $\{ \alpha_i : i \in I \}$ form a basis
of $\mathfrak{h}^*$.
\end{itemize}
The simple root vectors $e_i$ and $f_i$ generate ``opposite'' nilpotent
Lie subalgebras of $\mathfrak{g}$, denoted by $\mathfrak{n}^+$ and
$\mathfrak{n}^-$ respectively. (In our case, these are the strictly upper
and strictly lower triangular matrices, generated by $\{ e_i, f_i :
i \in [n] \}$ via the commutator bracket $[X,Y] := XY-YX$.)
Moreover, each $e_i$ is a simultaneous eigenvector for the adjoint action
of all of $\mathfrak{h}$. For instance in
$\mathfrak{sl}_{n+1}(\mathbb{C})$, we have
\[
[h, e_i] = [{\rm diag}(h_{jj})_j, E_{i,i+1}] = (h_{ii} - h_{i+1,i+1})
E_{i,i+1} = \alpha_i(h) e_i, \quad \forall h \in \mathfrak{h}.
\]
In addition, for any $i \neq j \in [n+1]$ we have $[h, E_{ij}] =
(\varepsilon_i - \varepsilon_j)(h) E_{ij}$. These nonzero functionals
$\varepsilon_i - \varepsilon_j,\ i \neq j$ are called the roots, and they
are nonnegative/nonpositive integer linear combinations of the simple
roots $\alpha_i$; e.g.\ if $i<j$ then $\varepsilon_i - \varepsilon_j =
\alpha_i + \alpha_{i+1} + \cdots + \alpha_{j-1}$.
Thus, $\mathfrak{n}^\pm$ are direct sums of one-dimensional root spaces
$\mathbb{C} E_{ij}$ with pairwise distinct roots (this holds for all
semisimple $\mathfrak{g}$).
We also write $\Delta = \{ \varepsilon_i - \varepsilon_j : i < j \in
[n+1] \}$ for the positive roots of $\mathfrak{g}$ -- note, this differs
from the Lie theory convention where $\Delta$ denotes all roots.

\subsection{Verma and finite-dimensional modules}

Denote the universal enveloping algebra of $\mathfrak{g}$ by
\[
U \mathfrak{g} := T(\mathfrak{g}) / \langle x \otimes y - y \otimes x -
[x,y] : x,y \in \mathfrak{g} \rangle.
\]
Recall, this is a unital associative $\mathbb{C}$-algebra and 
$\mathfrak{g} \hookrightarrow U \mathfrak{g}$. Moreover, the
multiplication map ${\rm mult} : U \mathfrak{n}^- \otimes U \mathfrak{h}
\otimes U \mathfrak{n}^+ \to U \mathfrak{g}$ is a $\mathbb{C}$-vector
space isomorphism.

Representations of
$\mathfrak{g}$ are precisely (left) $U \mathfrak{g}$-modules. An
important class of these consists of the \textit{Verma modules}
$M(\lambda)$ for all \textit{weights} $\lambda \in \mathfrak{h}^*$,
defined via
\[
M(\lambda) := \frac{U \mathfrak{g}}{U \mathfrak{g} \cdot \mathfrak{n}^+ +
\sum_{i \in I} \left( U \mathfrak{g} \cdot (h_i - \lambda(h_i)) \right)}.
\]
Thus from above, $M(\lambda) \cong U \mathfrak{n}^-$ as free rank-one $U
\mathfrak{n}^-$-modules, independent of $\lambda \in \mathfrak{h}^*$.
Moreover, the image of $1$ in $U \mathfrak{g}$, denoted by $m_\lambda$,
is a \textit{weight vector} (simultaneous eigenvector) for the
action of $\mathfrak{h}$ via $h \cdot m_\lambda = \lambda(h) m_\lambda$.
Thus the $\mathfrak{h}$-weight of e.g.\ $f_i^r m_\lambda$ is $\lambda - r
\alpha_i$, for $i \in I$ and $r \in \mathbb{N}$.

This brings us to the \textit{character} of a Verma module. Fix an
enumeration of the positive roots, say $\beta_1, \dots, \beta_k$; this
yields an ordered basis $(f_{\beta_1}, \dots, f_{\beta_k})$ of
$\mathfrak{n}^-$. Now by the above and the Poincar\'e--Birkhoff--Witt
(PBW) theorem, the words
\[
{\bf f}_\beta^{\bf m} := f_{\beta_1}^{m_1} \cdots f_{\beta_k}^{m_k},
\quad m_1, \dots, m_k \in \mathbb{N}
\]
form a $\mathbb{C}$-basis of $U \mathfrak{n}^-$. These words also
satisfy $[h, {\bf f}_\beta^{\bf m}] = -\sum_{r=1}^k m_r \beta_r(h) {\bf
f}_\beta^{\bf m}$ for all $h \in \mathfrak{h}$; i.e., ${\bf f}_\beta^{\bf
m}$ has $\mathfrak{h}$-weight $-\sum_{r=1}^k m_r \beta_r$. Similar
to above, the $\mathfrak{h}$-weight of ${\bf f}_\beta^{\bf m} m_\lambda$
equals $\lambda - \sum_{r=1}^k m_r \beta_r$. Thus via the isomorphism
$M(\lambda) \cong U(\mathfrak{n}^-)$, each weight space multiplicity
\[
\dim M(\lambda)_\mu = \dim U(\mathfrak{n}^-)_{\mu-\lambda} =:
K(\lambda-\mu)
\]
equals the number of ways in which to write $\lambda-\mu$ as a sum of
positive roots. (See e.g.\ Table~\ref{Ttable} below for some explicit
computations.) This map $K$ is the (usual) \textit{Kostant partition
function}. Thus we come to the character of $M(\lambda)$ as the
$e^\lambda$-shift of the generating function of $K$:
\begin{equation}\label{EKPF}
\ch M(\lambda) = \sum_{\beta \in \mathfrak{h}^*} K(\beta)
e^{\lambda-\beta} = \frac{e^\lambda}{ \displaystyle \prod_{r=1}^k \left(1 -
e^{-\beta_r}\right)}, \quad \lambda \in \mathfrak{h}^*.
\end{equation}

Having discussed (notation for) Verma modules, we turn to another
important class of $\mathfrak{g}$-representations: the finite-dimensional
modules. By Weyl's theorem, each of these is a direct sum of simple
modules, so it suffices to understand the latter. Recall that a weight
$\lambda \in \mathfrak{h}^*$ is said to be \textit{integral} if
$\lambda(h_i) \in \mathbb{Z}$ for all $i \in I$; these form a lattice
that is denoted in \cite{Hum} and in \cite{HMMS} by $\Lambda$. Within it
are the \textit{dominant integral weights}:
\begin{equation}\label{Eweightlattice}
\Lambda := \{ \lambda \in \mathfrak{h}^* :\lambda(h_i) \in
\mathbb{Z}\ \forall i \in I \} \quad \supset \quad \Lambda^+ := \{
\lambda \in \Lambda : \lambda(h_i) \geq 0\ \forall i \in I \}.
\end{equation}

Now the ``theorem of the highest weight'' says that simple
finite-dimensional $\mathfrak{g}$-modules are -- up to isomorphism -- in
bijection with dominant integral weights. More precisely, this bijection
sends $\lambda \in \Lambda^+$ to the quotient module
\[
V(\lambda) := M(\lambda) / \sum_{i \in I} U \mathfrak{g} \cdot
f_i^{\lambda(h_i) + 1} m_\lambda,
\]
and this is finite-dimensional and simple. Moreover, the celebrated
Weyl character formula says this module has character given by the
\textit{Schur polynomial} $s_\lambda$, and in the above notation it
equals
\begin{equation}\label{EWCF}
\ch V(\lambda) = \sum_{w \in W} (-1)^{\ell(w)} \frac{e^{w(\lambda+\rho) -
\rho}}{ \displaystyle \prod_{r=1}^k (1 - e^{-\beta_r})}
= \sum_{w \in W} (-1)^{\ell(w)} \ch M(w \bullet \lambda), \quad
\lambda \in \Lambda^+.
\end{equation}

Here, $\rho = \frac{1}{2} \sum_{r=1}^k \beta_r$ is the half-sum of the
positive roots, $w \bullet \lambda := w(\lambda + \rho) - \rho$, and $W$
is the Weyl group, which is the finite group of orthogonal
transformations of $\mathfrak{h}^*$ generated by the simple reflections
$s_{\alpha_i}$ -- with associated length function $\ell$.
E.g.\ for $\mathfrak{sl}_{n+1}(\mathbb{C})$, $W$ is the symmetric group
$S_{n+1}$, generated by the simple transpositions $(i\ i+1)$ for $1 \leq
i \leq n$; and a weight $\lambda = (\lambda_1, \dots, \lambda_n,
\lambda_{n+1})^T$ is dominant integral if and only if $\lambda(h_i) =
\lambda_i - \lambda_{i+1} \in \mathbb{N}$ for all $i \leq n$, i.e.\
$\lambda - \lambda_{n+1} (1,\dots,1)$ is a partition.

\subsection{Parabolic Verma modules}

Finally, we introduce the parabolic Verma modules, which are a natural
family of universal highest weight modules that interpolate between the
Verma modules $M(\lambda)$ for all $\lambda \in \mathfrak{h}^*$ and the
simple modules $V(\lambda)$ for $\lambda \in \Lambda^+$. The key fact
used here is that if $\lambda \in \mathfrak{h}^*$ and $i \in I$ are such
that $\lambda(h_i) \in \mathbb{N}$, then inside the Verma module
$M(\lambda)$, the vector $f_i^{\lambda(h_i)+1} m_\lambda$ is a highest
weight vector -- that is, it is killed by all of $\mathfrak{n}^+$ and has
$\mathfrak{h}$-weight $s_i \bullet \lambda := \lambda - (\lambda(h_i) + 1)
\alpha_i$. In particular, $U  \mathfrak{g} \cdot f_i^{\lambda(h_i) + 1}
m_\lambda \cong M(s_i \bullet \lambda)$.

Given a subset $J \subseteq I$ of (indices of) simple roots,
define the \textit{$J$-dominant integral weights} to be
\begin{equation}\label{ELambdaJplus}
\Lambda^+_J := \{ \lambda \in \mathfrak{h}^* : \lambda(h_i) \in
\mathbb{N}\ \forall i \in J \},
\end{equation}
and for each $\lambda \in \Lambda^+_J$ define the \textit{parabolic Verma
module}
\begin{equation}\label{Egvm}
M(\lambda,J) := M(\lambda) / \sum_{i \in J} U \mathfrak{g} \cdot
f_i^{\lambda(h_i) + 1} m_\lambda.
\end{equation}

Two ``extremal'' special cases of these modules come from the two
extremal values for $J$:
\begin{itemize}
\item If $J = \emptyset$, then $\Lambda^+_J = \mathfrak{h}^*$ and
$M(\lambda,J) = M(\lambda)$.
\item If $J = I$, then $\Lambda^+_J = \Lambda^+$ and $M(\lambda,J) =
V(\lambda)$.
\end{itemize}

Thus, parabolic Verma modules subsume both Verma modules and
finite-dimensional simple $\mathfrak{g}$-modules. In addition, the
Weyl character formula also extends to these modules:
\[
\ch M(\lambda,J) = \sum_{w \in W_J} (-1)^{\ell(w)} \frac{e^{w \bullet
\lambda}}{ \displaystyle \prod_{r=1}^k (1 - e^{-\beta_r})}
= \sum_{w \in W_J} (-1)^{\ell(w)} \ch M(w \bullet \lambda), \quad J
\subseteq I, \ \lambda \in \Lambda^+_J.
\]
(Here $W_J$ is the parabolic Weyl subgroup, generated by the simple
reflections $\{ s_{\alpha_i} : i \in J \}$.)
But even more is true: the Weyl character formula~\eqref{EWCF} is the
combinatorial shadow -- via taking the Euler characteristic -- of the BGG
resolution of $V(\lambda)$:
\[
0 \longrightarrow \bigoplus_{w \in W : \ell(w) = k} M(w \bullet \lambda)
\longrightarrow \ \cdots \ \longrightarrow \bigoplus_{w \in W : \ell(w) =
1} M(w \bullet \lambda) \longrightarrow M(\lambda) \longrightarrow
V(\lambda) \longrightarrow 0.
\]
(Trivially, the same 1-step resolution holds for every Verma module.) In fact such a resolution turns out to exist even more generally
-- for all parabolic Verma modules; see \cite{Hum} for details.\bigskip

Given these multiple ways in which parabolic Verma modules have the same
fundamental properties as Vermas and finite-dimensional simples, it is
natural to ask if their characters are always log-concave.   Indeed, the
extreme cases of $M(\lambda), \ \lambda \in \mathfrak{h}^*$ and
$V(\lambda), \ \lambda \in \Lambda^+$ were proven in \cite{HMMS}. The
motivating goal of this work is to answer this question affirmatively.
%}}}

%{{{1 Section 4 - Lorentzianity of normalized shifted characters of parabolic Vermas
\section{Lorentzianity of normalized shifted characters of parabolic
Vermas}\label{Slorentzian}

In this section we show Theorems~\ref{T1} and~\ref{T2}. We first provide
the unfamiliar reader with a pathway to go from parabolic Vermas to
Kostant partition functions.

\subsection{From parabolic Verma modules to restricted Kostant partition
functions} 

An appropriate notion to study characters of parabolic Verma modules is
that of restricted Kostant partition functions (KPFs). This
extends~\eqref{EKPF} which rewrote all Verma module characters as shifts
of the generating function of the (usual) Kostant partition function.
Thus, we first explain how restricted KPFs naturally encode parabolic
Verma characters. As above, readers familiar with some but not all of the
material can skip the relevant subsections, only glancing at them for the
notation used later.

The \textbf{first step} in showing the discrete
log-concavity~\eqref{EdiscLC} of all $\ch M(\lambda,J)$ is to note that
every parabolic Verma module is obtained via parabolic induction from a
finite-dimensional simple module over a semisimple Lie subalgebra.
Namely, define $\mathfrak{g}_J$ to be the Lie subalgebra of
$\mathfrak{g}$ generated by the Chevalley generators $\{ e_i, f_i : i \in
J \}$. Then define the parabolic Lie subalgebra
\[
\mathfrak{p}_J := \mathfrak{g}_J + \mathfrak{h} + \mathfrak{n}^+.
\]
Notice that if $\lambda \in \Lambda^+_J$, then $\lambda(h_i) \in
\mathbb{N}$ for all $i \in J$; thus one forms the finite-dimensional
$\mathfrak{g}_J$-module $V_J(\lambda)$, generated by a highest weight
vector $v_\lambda$. This in fact has a $\mathfrak{p}^+_J$-module
structure via
\[
h \cdot v_\lambda = \lambda(h) v_\lambda\ \forall h \in \mathfrak{h};
\qquad \mathfrak{n}^+ \cdot v_\lambda = 0.
\]
Now it is known that the parabolic Verma module is the induction of this
$\mathfrak{g}_J$-integrable module:
\[
M(\lambda,J) \cong {\rm Ind}^{U \mathfrak{g}}_{U \mathfrak{p}_J}
V_J(\lambda).
\]

As above, let $\Delta$ denote the positive roots of $\mathfrak{g}$, i.e.\
the roots of $\mathfrak{n}^+$; and let $\Delta_J$ denote the (positive)
roots of $\mathfrak{n}^+_J$ -- these are also the roots of
$\mathfrak{n}^+$ that are $\mathbb{N}$-linear combinations of $\{
\alpha_i : i \in J \}$. Now define
\begin{equation}\label{Eparabolic}
\mathfrak{u}^-_J := \bigoplus_{\beta \in \Delta \setminus \Delta_J}
\mathfrak{n}^-_{-\beta};
\end{equation}
this is a Lie subalgebra of $\mathfrak{g}$ (in fact of $\mathfrak{n}^-$),
spanned by all root spaces $\mathfrak{g}_{-\beta} =
\mathfrak{n}^-_{-\beta}$ such that $\beta$ is an $\mathbb{N}$-linear
combination of simple roots $\alpha_i$ with at least one $i \not\in J$.
For example, in our case of $\mathfrak{g} =
\mathfrak{sl}_{n+1}(\mathbb{C})$,
\[
\mathfrak{u}^-_J = {\rm span}_{\mathbb{C}} \{ E_{ij} : i>j, \ \{ i-1,
\dots, j+1, j \} \nsubseteq J \}.
\]

The PBW theorem gives a vector space isomorphism:
\begin{equation}\label{Egvm-tensor}
M(\lambda, J) \cong_{\mathbb{C}} U (\mathfrak{u}^-_J)
\otimes_{\mathbb{C}} V_J(\lambda),
\end{equation}
and since characters are multiplicative across tensor products, this
yields
\[
\ch M(\lambda,J) = \ch U(\mathfrak{u}^-_J) \cdot \ch V_J(\lambda).
\]

The \textbf{second step} is to note that the latter factor is indeed
log-concave along type $A$ root directions. Indeed, partition the Dynkin
subdiagram $J \subseteq I = [n]$ into disjoint connected components $J =
J_1 \sqcup \cdots \sqcup J_l$; thus each $J_r$ is a contiguous
subinterval, and so $\mathfrak{g}_{J_r} \cong
\mathfrak{sl}_{|J_r|+1}(\mathbb{C})$ for all $r$. The following
decompositions into pairwise commuting summands/factors are now standard:
\[
\mathfrak{g}_J = \oplus_{r=1}^l \mathfrak{g}_{J_r}, \qquad
U (\mathfrak{g}_J) = \bigotimes_{r=1}^l U (\mathfrak{g}_{J_r}),
\]
and the respective Cartan subalgebras satisfy the same relations. Thus
$\lambda = \oplus_{r=1}^l \lambda|_{\mathfrak{h}_r} = (\lambda_1, \dots,
\lambda_l)$, say. Then we also have vector space isomorphisms, across
which $\ch(\cdot)$ is multiplicative:
\begin{equation}\label{Ecommute}
M_{\mathfrak{g}_J}(\lambda) \cong_{\mathbb{C}} \bigotimes_{r=1}^l 
M_{\mathfrak{g}_{J_r}}(\lambda_r), \qquad
V_J(\lambda) \cong_{\mathbb{C}} \bigotimes_{r=1}^l V_{J_r}(\lambda_r);
\end{equation}
moreover, the characters of the tensor factors in the second isomorphism
are polynomials in disjoint sets of variables, as is explained below.

The \textbf{third step} in this part is to compute $\ch
U(\mathfrak{u}^-_J)$. We claim more strongly that it is given by a
restricted Kostant partition function for all $J \subsetneq I$. (Note
that $\mathfrak{u}^-_I = 0$, so $U(\mathfrak{u}^-_I) = \mathbb{C}$.)
To show the claim, fix $J \subsetneq I$ and enumerate $\Delta
\setminus \Delta_J = \{ \beta_1, \dots, \beta_p \}$.
By~\eqref{Eparabolic} and the PBW theorem,
\begin{equation}\label{ErestKPF}
\ch U(\mathfrak{u}_J^-) = \frac{1}{\prod_{r=1}^p (1 - e^{-\beta_r})}.
\end{equation}

In other words, $\dim U(\mathfrak{u}_J^-)_\mu$ is the number of ways to
write $-\mu$ as an $\mathbb{N}$-linear combination of $\beta_1, \dots,
\beta_p$. This is precisely a restricted Kostant partition function
(KPF), as we now define.

\begin{definition}\label{Dkpf}
Let $G$ be a loopless multigraph on the vertices $[n+1]$ with edges
directed from smaller to larger vertices.
Denote by $K_G$ the \textit{restricted Kostant partition function (KPF)},
which takes a vector $v = (v_1, \dots, v_n, v_{n+1}) \in
\mathbb{Z}^{n+1}$ to the number $K_G(v)$ of ways to write $v$ as a sum
of the positive type $A$ roots $\varepsilon_i - \varepsilon_j \in
\mathbb{Z}^{n+1}$ corresponding to edges $(i,j)$ in $G$ (with
multiplicity). For instance if $G$ is the complete (directed simple)
graph, we get the usual/unrestricted KPF $K(v)$.
\end{definition}

Now the claim is shown as follows. Let $\mu = -\sum_{i=1}^n l_i
\alpha_i \in \mathfrak{h}^*$ for some $l_i \in \mathbb{C}$. Then the
space $U(\mathfrak{u}^-_J)_\mu = 0$ unless all $l_i \in \mathbb{N}$. More
strongly, if we define the graph $G_J$ on $[n+1]$ by only including those
edges $i \to j$ for which $i < j$ and $\{ i, i+1, \dots, j-1 \}
\nsubseteq J$, then the discussion around~\eqref{ErestKPF} implies that
(a)~there are exactly $p$ such edges (and no multi-edges) in $G_J$,
say $i_r \to j_r$ for $r \in [p]$;
(b)~up to relabelling, $\beta_r = \varepsilon_{i_r} - \varepsilon_{j_r}$
for all $r$;
and (c)~we have that
\begin{align*}
-\mu = \sum_{i=1}^n l_i \alpha_i = &\ l_1 \varepsilon_1 + (l_2 - l_1)
\varepsilon_2 + \cdots + (l_n - l_{n-1}) \varepsilon_n - l_n
\varepsilon_{n+1}\\
\implies \dim U(\mathfrak{u}^-_J)_\mu = &\ K_{G_J}(l_1, l_2-l_1, \dots,
l_n-l_{n-1},-l_n).
\end{align*}

\subsection{Completing the proof}

We now prove Theorems~\ref{T1} and~\ref{T2}; we follow the approach
in \cite[Proposition~13]{HMMS}. Recall the normalization operator
in~\eqref{Enormal}; and given a tuple $\beta \in \mathbb{N}^m$, let
$\partial^\beta$ denote the $\beta$th partial derivative of polynomials
or power series $p(x) \in \mathbb{R}[[x_1, \dots, x_m]]$, sending a
monomial $x^\mu$ to $\frac{\mu!}{(\mu-\beta)!} x^{\mu-\beta}$. Thus if
$p(x) := \sum_{\mu \geq {\bf 0}} c_\mu x^\mu$, then we have
\begin{equation}\label{Enormal2}
\partial^\beta N(p(x)) = \sum_{\mu \geq \beta} \frac{c_\mu}{\mu!}
\frac{\mu!}{(\mu-\beta)!} x^{\mu-\beta} = N(p(x) \cdot x^{-\beta}).
\end{equation}

%We now complete the proofs of the first two main results of the paper.

\begin{proof}[Proof of Theorem~\ref{T2}]
In \cite{HMMS}, to show Lorentzianity the authors worked with
characters of Verma modules over $\mathfrak{sl}_{n+1}(\mathbb{C})$, and
then translated this into flow polytopes over the complete simple graph
on $[n+1]$. As we now work with arbitrary multigraphs, we adjust the
argument.

Suppose $G$ contains $m_{ij} \geq 0$ edges $i \to j$, for each pair $i<j$
in $[n+1]$. Then the generating function of $K_G(\cdot)$ is
\[
\chG(x_1, \dots, x_{n+1}) = \prod_{j>i} (1 + x_j x_i^{-1} + x_j^2
x_i^{-2} + \cdots)^{m_{ij}};
\]
this is well-defined in the power series ring $\mathbb{R}[[
\frac{x_2}{x_1}, \dots, \frac{x_{n+1}}{x_n} ]]$.

We now show that the expression $N(x^\delta \cdot \chG(x))$ is Lorentzian
for all $\delta \in \mathbb{N}^{n+1}$. Note that only the terms $x^\mu$
with $\mu \geq -\delta$ (coordinatewise) contribute to this expression.
Now choose any positive integers $n_{ij}$ for $i>j$ such that $\delta_i
\leq \sum_{j>i} n_{ij} m_{ij} =: \beta_i$ for all $i \in [n]$, and
compute
\begin{equation}\label{Enormal3}
N(x^\delta \cdot \chG(x)) = N \left( x^\delta \prod_{j>i}
(x_j^{n_{ij}} + x_i x_j^{n_{ij}-1} + \cdots + x_i^{n_{ij}})^{m_{ij}}
\cdot x^{-\beta} \right).
\end{equation}
Define the homogeneous polynomial
\[
p(x) := x^\delta \prod_{j>i} (x_j^{n_{ij}} + x_i x_j^{n_{ij}-1} + \cdots
+ x_i^{n_{ij}})^{m_{ij}}.
\]
Since each polynomial factor in this product (without the exponent of
$m_{ij}$) as well as $x^\delta$ has a Lorentzian normalization, $N(p(x))$
is Lorentzian by Theorem~\ref{thm:BHdoubleLC} (3). As taking partial
derivatives preserves the Lorentzian property, \eqref{Enormal2} yields
that $N(x^\delta \cdot \chG(x))$ is also Lorentzian,
via~\eqref{Enormal3}, 
Again using Theorem~\ref{thm:BHdoubleLC}, we obtain both the continuous
log-concavity of $N(x^\delta \cdot \chG(x))$ and the discrete
log-concavity of $\chG(x)$.
\end{proof}

\begin{proof}[Proof of Theorem~\ref{T1}]
From~\eqref{Egvm-tensor} and~\eqref{Ecommute} we know that
\[
\ch M(\lambda, J) = \ch U(\mathfrak{u}^-_J) \prod_{r=1}^l \ch
V_{J_r}(\lambda_r),
\]
and as discussed above, $\ch U(\mathfrak{u}^-_J)$ is the generating
function of the restricted KPF $K_{G_J}$ (with $G_J$ introduced after
Definition~\ref{Dkpf}). We now follow the proof of Theorem~\ref{T2} (as
one cannot directly apply it). Let $G_J$ contain $m_{ij} \in \{ 0, 1 \}$
edges $i \to j$ for $i<j$ in $[n+1]$. Given $\delta \in
\mathbb{N}^{n+1}$, choose $n_{ij}$ and define $\beta_i$ as in the
preceding proof, and compute as in~\eqref{Enormal3}:
\[
N(x^\delta \cdot \ch M(\lambda,J)) = N \left( x^\delta \prod_{j>i}
(x_j^{n_{ij}} + x_i x_j^{n_{ij}-1} + \cdots + x_i^{n_{ij}})^{m_{ij}}
\cdot \prod_{r=1}^l \ch V_{J_r}(\lambda_r) \cdot x^{-\beta} \right).
\]
The factors in the second product have Lorentzian normalizations by
\cite[Theorem~3]{HMMS}, as do the factors in the first product as well as
$x^\delta$. As in the proof of Theorem~\ref{T2}, it follows
using~\eqref{Enormal2} and Theorem~\ref{thm:BHdoubleLC} (3) that
$N(x^\delta \cdot \ch M(\lambda,J))$ is also Lorentzian. In turn, this
yields both the continuous log-concavity of $N(x^\delta \cdot \ch
M(\lambda,J))$ and the discrete log-concavity of $\ch M(\lambda,J)$
 by Theorem~\ref{thm:BHdoubleLC}.
\end{proof}
%}}}

%{{{1 Section 3 - Alternate approach to discrete log-concavity, via flow polytopes
\section{Alternative approach to discrete log-concavity, via flow polytopes}\label{Sflow}

We now explain an alternative way of proving the discrete
log-concavity in Theorem~\ref{T2}: using flow polytopes. As above, we
start with a quick introduction to flow polytopes; the interested reader
may see \cite{vol_ehr_flow_polys} for a more thorough and general
treatment.

\subsection{Flow polytopes and Kostant partition functions}

By convention, we will use \emph{graph} to mean a loopless directed
finite multigraph on a labeled vertex set $[n+1]$ with edges directed from
$i$ to $j$ when $i<j$ (hence acyclic).

Let $G$ be a graph on vertex set $[n+1]$. For $a=(a_1,\ldots,a_n)\in
\mathbb{R}^n$, an \emph{$a$-flow} on $G$ is a function
$f\colon E(G) \to \mathbb{R}_{\geq 0}$ such that the
\emph{flow conservation} condition
\[
\sum_{e=(i',i)\in E(G)}f(e) + a_i = \sum_{e=(i,i')\in E(G)} f(e)
\]
holds for each $i\in [n]$. Note that summing these $n$ equations and
simplifying gives
\[
\sum_{e=(i,n+1)\in E(G)}f(e) = -\sum_{i=1}^n a_i.
\]
In other words, flow conservation at $i=n+1$ is implied,
when one completes the flow vector $a$ to include an additional coordinate
such that the $n+1$ coordinates sum to zero.

\begin{definition}
For any $a \in \mathbb{R}^n$, the \emph{flow polytope} $\mathcal{F}_G(a)$
of $G$ is the set of $a$-flows on $G$. 
\end{definition}

We denote by $k$ the number of edges of $G$ (with multiplicity).
By fixing an integral equivalence between
the affine span of $\mathcal{F}_G(a)$ and
$\mathbb{R}^{k-n}$, we may view
$\mathcal{F}_G(a)$ as a full-dimensional polytope in
$\mathbb{R}^{k-n}$ instead of a polytope in $\mathbb{R}^{E(G)}$
when convenient.

We now explain the connection between Ehrhart theory of integral flow
polytopes and the restricted Kostant partition functions. Given $G$ as
above, let $A_G$ be the $(n+1)\times k$ matrix with a column
$\varepsilon_i-\varepsilon_j$ for each edge $e=(i,j)$ in $G$ (with
multiplicity). A straightforward check shows that for any completed flow
vector $\widetilde{a} = (a_1,\ldots,a_n,-\sum_{i=1}^n a_i)$,
\[
\mathcal{F}_G(a_1,\ldots,a_n) = \left\{f\in \mathbb{R}^k_{\geq 0} : 
A_G f = \widetilde{a}^T \right\}.
\]
In particular, the number of integer points in
$\mathcal{F}_G(a_1,\ldots,a_n)$ is exactly
$K_G\left(a_1,\dots,a_n,-\sum_{i=1}^na_i\right)$.

\begin{example}
If $G$ is the complete graph on vertices $[n+1]$, then $k=\binom{n+1}{2}$
is the number of type $A$ positive roots, and the number of integer
points in $\mathcal{F}_G(a_1, \dots, a_n)$ equals the usual Kostant
partition function $K(a_1, \dots, a_n, -\sum_{i=1}^n a_i)$.
\end{example}

Remarkably, volumes of flow polytopes are also given by Kostant partition
functions. The following formula was proved by Baldoni and Vergne in
\cite{BV} using residue techniques. It was subsequently reproved by
Postnikov and Stanley in unpublished work \cite{r_stan_slides}, and again
by M\'esz\'aros and Morales in \cite{vol_ehr_flow_polys} via an explicit
subdivision. We use the notation and formulation of
\cite{vol_ehr_flow_polys} below. 

Recall the dominance order (or weak majorization) on $\mathbb{R}^n$ is
given by: $a$ dominates $b$ if $a_1+\cdots+a_i \geq b_1+\cdots+b_i$ for
each $i\in[n]$.

\begin{theorem}[Lidskii volume formula {\cite[Theorem
38]{BV}}]\label{thm:lidskii}
Let $G$ be a graph on $[n+1]$ with $k$ edges (directed from smaller to
larger vertices). Suppose that each vertex $i\in [n]$ has at least one
outgoing edge. Then for any $a_1,\ldots,a_n\geq 0$,
\[
\mathrm{Vol}(\mathcal{F}_G(a_1,\ldots,a_n)) = \sum_r (k-n)! \,
K_{G}\left(r_1-o^G_1, \ldots, r_n -
o^G_n,0\right)\frac{a_1^{r_1}}{r_1!} \cdots \frac{a_n^{r_n}}{r_n!}
\]
where $o^G_i=\mathrm{outdeg}_G(i)-1$, and the sum is over weak
compositions $r=(r_1,r_2,\ldots,r_n)$ of $k-n$ that are $\geq$
$o^G := (o^G_1,\ldots,o^G_n)$ in dominance order.
\end{theorem}

From the flow conservation condition, one can observe that whenever $a$
does not dominate the zero vector $G$ admits no $a$-flows. In this case, 
clearly $K_G(a)=0$. Hence the condition that $r$ dominates $o^G$ above
can be dropped from the sum if desired. Also note the additional
requirement above that $a_1,\ldots,a_n\geq 0$ for this volume formula,
which is not required in the definition of flow polytopes.

We conclude this foray into flow polytopes with a useful property
required later: the flow polytopes considered in Theorem
\ref{thm:lidskii} admit a Minkowski sum decomposition into simpler flow
polytopes (see for instance \cite[Proposition 2.1]{vol_ehr_flow_polys}).

\begin{proposition}\label{prop:minksumdecomp}
For any graph $G$ on $[n+1]$ and any $a_1,\ldots,a_n\geq 0$,
\[
\mathcal{F}_G(a)=\sum_{i=1}^n a_i
\mathcal{F}_G(\varepsilon_i).
\]
\end{proposition}

\subsection{Mixed volumes of polytopes and the Alexandrov--Fenchel
inequality}

Let $P_1,\ldots, P_n$ be polytopes in $\mathbb{R}^k$ and fix real weights
$a_1, \ldots, a_n \geq 0$. Set $P$ to be the Minkowski sum
\[
P = a_1 P_1 + \cdots + a_n P_n.
\]

By classical results on convex sets (see for instance \cite[Theorem
5.2.39]{convexity_book}) the volume $\mathrm{Vol}(P)$ of $P$ is a
homogeneous polynomial of degree $k$ in $a_1, \ldots, a_n$:
\[
\mathrm{Vol}(P) = \sum_{s_1=1}^n \sum_{s_2=1}^n \cdots \sum_{s_k=1}^n
V(P_{s_1}, \dots, P_{s_k}) a_{s_1} \cdots a_{s_k}.
\]

The coefficients $V(P_{s_1}, \dots, P_{s_k})$ are uniquely determined by
requiring that they be symmetric up to permutations of arguments. The
number $V(P_{s_1}, \dots, P_{s_k})$ is called the \emph{mixed volume}
of $P_{s_1},\ldots,P_{s_k}$.

We will represent mixed volumes with the notation
\[
V(P_1^{r_1},\ldots,P_n^{r_n}) :=
V\left(\underbrace{P_1,\ldots,P_1}_{r_1}, \ \ldots ,\ 
\underbrace{P_n,\ldots,P_n}_{r_n}\right).
\]
Then
\begin{align*}
\mathrm{Vol}(P)&=\sum_{\substack{r_1,\ldots,r_n \geq
0\\r_1+\cdots+r_n=k}} \binom{k}{r_1,\ldots,r_n} V\left(P_1^{r_1}, \ldots,
P_n^{r_n}\right)a_1^{r_1}\cdots a_n^{r_n}\\
&=\sum_{\substack{r_1,\ldots,r_n\geq 0\\r_1+\cdots+r_n=k}}
k!\, V\left(P_1^{r_1},\ldots,P_n^{r_n}\right)\frac{a_1^{r_1}}{r_1!}\cdots
\frac{a_n^{r_n}}{r_n!}.
\end{align*}

We will derive log-concavity of the characters of parabolic Verma modules
from the Alexandrov--Fenchel inequalities, a fundamental result in convex
geometry proved independently by Alexandrov in \cite{alexandrov} and
Fenchel in \cite{fenchel1,fenchel2}. These inequalities state that mixed
volumes are discretely log-concave, and have been used to derive many
instances of discrete log-concavity in combinatorics (for a survey, see
\cite{stanley_lc_survey}).

\begin{theorem}[Alexandrov--Fenchel inequalities]\label{thm:afinequality}
Fix $i,j\in [n]$ with $i<j$. Then for any integers $r_1,\ldots,r_n \in
\mathbb{N}$ with $r_i, r_j \geq 1$,
\[
V(P_1^{r_1},\ldots, P_n^{r_n})^2\geq
V(P_1^{r_1},\ldots,P_{i}^{r_i+1},\ldots,P_{j}^{r_j-1},\ldots, P_n^{r_n})
V(P_1^{r_1},\ldots,P_{i}^{r_i-1},\ldots,P_{j}^{r_j+1},\ldots, P_n^{r_n}).
\]
\end{theorem}

The equality conditions of Theorem \ref{thm:afinequality} remain a major open problem, with recent advancements made in \cite{geom_af_inequalities,af_equality_complexity}.

\subsection{Discrete log-concavity of restricted Kostant partition
functions}

We can finally finish the alternative proof of the first part of
Theorem~\ref{T2}. We need one last intermediate result. For a graph $G$,
recall the numbers $o^G_i=\mathrm{outdeg}_G(i)-1$. The following result
is an easy consequence of the Lidskii volume formula.

\begin{proposition}\label{prop:mixedvolumekostant}
Let $G$ be a graph on vertices $[n+1]$ with $k$ edges and at least one
outgoing edge from each vertex $i\in [n]$. Then for any weak composition
$r=(r_1,\ldots,r_n)$ of $k-n$,
\[
V \left(\mathcal{F}_G(\varepsilon_1)^{r_1}, \ldots,
\mathcal{F}_G(\varepsilon_n)^{r_n} \right) = K_{G} \left( r_1-o^G_1,
\ldots, r_n - o^G_n, 0 \right).
\]
\end{proposition}

\begin{proof}
For each $i\in [n]$, set $P_i = \mathcal{F}_G(\varepsilon_i)$ viewed as a
polytope in $\mathbb{R}^{k-n}$. For any $a_1,\ldots,a_n\geq 0$,
Proposition \ref{prop:minksumdecomp} implies
\begin{align*}
\mathrm{Vol}(\mathcal{F}_G(a_1,\ldots,a_n)) &=
\mathrm{Vol}(a_1P_1+\cdots+a_nP_n)\\
&=\sum_{\substack{r_1,\ldots,r_n\geq 0\\r_1+\cdots+r_n=k-n}}
(k-n)! \, V\left(P_1^{r_1},\ldots,P_n^{r_n}\right)\frac{a_1^{r_1}}{r_1!}
\cdots \frac{a_n^{r_n}}{r_n!}.
\end{align*}

From Theorem \ref{thm:lidskii} and the remark thereafter, we obtain
\begin{align*}
\mathrm{Vol}(\mathcal{F}_G(a_1,\ldots,a_n))
&=\sum_{\substack{r_1,\ldots,r_n\geq 0\\r_1+\cdots+r_n=k-n}}
(k-n)! \, K_{G}\left(r_1-o^G_1, \ldots, r_n -
o^G_n,0\right)\frac{a_1^{r_1}}{r_1!}\cdots \frac{a_n^{r_n}}{r_n!}.
\end{align*}

By Zariski density, comparing these two volume formulas yields
\[
V\left(P_1^{r_1},\ldots,P_n^{r_n}\right)=K_{G}\left(r_1-o^G_1, \ldots,
r_n - o^G_n,0\right).\qedhere
\]
\end{proof}

\noindent With the above analysis at hand, we can now show:

\begin{proof}[Proof of the discrete log-concavity in Theorem~\ref{T2}]
Fix any $v\in \mathbb{Z}^{n+1}$. First note that if $v_{n+1}\neq
-(v_1+\cdots+v_n)$, then both sides of the inequality are zero and
there is nothing to prove. We assume that $v_1+\cdots+v_{n+1}=0$.
Choose an integer $B>|v_1|+\cdots+|v_n|+|v_{n+1}|+n+1$. Let
$\widetilde{v}\in \mathbb{Z}^{B+1}$ denote $v$ with $B-n$ trailing zeros
appended. Set $H$ to be the graph on $[B+1]$ obtained by starting with
$G$ and connecting each new vertex $i>n+1$ to all smaller vertices.
Direct all edges from smaller to larger vertices as usual.

Observe that
\begin{align*}
K_H(\widetilde{v}) &= K_G(v),\\
K_H(\widetilde{v} + \widetilde{\varepsilon_i} - \widetilde{\varepsilon_j}) &= 
K_G(v + \varepsilon_i - \varepsilon_j),\mbox{ and }\\
K_H(\widetilde{v} - \widetilde{\varepsilon_i} + \widetilde{\varepsilon_j}) &= 
K_G(v - \varepsilon_i + \varepsilon_j)\mbox{ for all distinct } i,j\in [n+1].
\end{align*}	

For each $b \in [B]$, set $P_b=\mathcal{F}_H(\widetilde{\varepsilon_b})$
and $r_b=\widetilde{v}_b+o^H_b$. Note that the choice of $B$ implies $o^H_b\geq 1$ and $r_b\geq 0$ for each $b\in [B]$.
Hence the assumptions of Proposition \ref{prop:mixedvolumekostant} are met by $H$ and $r$, with its application yielding
\[
V\left(P_1^{r_1},\ldots,P_B^{r_B}\right) =
K_H\left(r_1-o^H_1,\ldots,r_B-o^H_B,0\right)
=K_H\left(\widetilde{v}\right)
=K_G\left(v\right).
\]
Applying the Alexandrov--Fenchel inequalities
(Theorem~\ref{thm:afinequality}) completes the proof.
\end{proof}

\subsection{Flow polytopes for parabolic Verma characters; products
of discretely log-concave polynomials}\label{Sadlc}

Given the preceding proof, it is natural to ask if this approach would
also help prove the Discrete Log-Concavity along type $A$ root directions
of $\ch M(\lambda,J)$ in Theorem~\ref{T1}. (For convenience, we refer to
this property as \textbf{ADLC} throughout this subsection.) Such an
alternative approach was indeed undertaken for the special case of Verma
modules in \cite{HMMS}.

In order for this approach to work for parabolic Vermas, a key step would
require proving that since the character of each tensor factor
in~\eqref{Egvm-tensor} satisfies ADLC, hence so does their product.
Stripping away the representation theory, the question becomes:

\textit{Is the set of multivariate homogeneous ADLC polynomials with
nonnegative coefficients closed under multiplication?}

One can further weaken this question, to assume that
\begin{itemize}
	\item[(a)] the coefficients are all nonnegative integers;
	\item[(b)] one of the two polynomials is a geometric series
	$x_i^k + x_i^{k-1} x_j + \cdots + x_j^k$ (hence trivially ADLC);
	\item[(c)] the exponents occurring in the other ADLC polynomial
	form an M-convex set;
\end{itemize}
and then ask if the two polynomials multiply to an ADLC output.

Unfortunately, this question is far from having a positive answer --
whence it is unclear how to proceed via Alexandrov--Fenchel in proving
the discrete log-concavity of parabolic Verma characters. We provide two
families of counterexamples here.

\begin{example}
(This example does not have the weakening~(b) above.)
June Huh communicated to us: let $p(x,y,z) = x^2 + 100 y^2 + z^2 + 10 xy
+ 10 yz + 10 xz$ and $q(x,y,z) = x+y+z$. Then $p,q$ have M-convex
supports and are ADLC, but $p \cdot q$ is not ADLC. One can also use $x^k
q(x,y,z)$ for $k \geq 1$ if polynomials of ``higher'' degree are desired.
\qed
\end{example}

These observations extend to the following result.

\begin{proposition}
Fix an integer $n \geq 2$ and a scalar $b \geq 13/2$, and let
\[
p(x_0, x_1, \dots, x_n) := b^2 x_0^2 + \sum_{i=1}^n x_i^2 + b \sum_{0
\leq i<j \leq n} x_i x_j.
\]
Then for every integer $k \geq 2$, and every finite M-convex subset $S
\subset \mathbb{N}^{n+1}$ with all elements having $\ell^1$-norm $k$ and
containing the points
\[
(k, {\bf 0}_n), \ (k-1,1, {\bf 0}_{n-1}), \
(k-1, 0, 1, {\bf 0}_{n-2}), \ (k-2,2, {\bf 0}_{n-1}), \
(k-2,1,1, {\bf 0}_{n-2}), \ (k-2,0,2, {\bf 0}_{n-2}),
\]
the polynomial $p \cdot q_S$ is not ADLC, even though $p,q_S$ are ADLC
with M-convex supports. Here, the homogeneous polynomial
$q_S(x_0, \dots, x_n) := \sum_{\mu \in S} x^\mu$.
\end{proposition}

\begin{proof}
The coefficients and exponents of $p$ may be graphically arranged in a
multi-dimensional array -- we depict it here for $n=2$:

\begin{center}
	\begin{tabular}{c c c}
		1 & $b$ & 1 \\
		$b$ & $b$ & \\
		$b^2$ & &
	\end{tabular}
\end{center}

\noindent It is easy to see that this 2-dimensional array is ADLC; similarly,
$p$ is ADLC for all $n \geq 2$. Also verify by inspection that $p$ has
M-convex support. Moreover, $q_S$ has M-convex support, hence
by~\cite{MTY19} has the SNP (saturated Newton polytope) property, meaning
that if one considers the lattice points that are the exponents in its
monomials, there are no ``internal gaps''. Hence all arithmetic
progressions in $S$ have corresponding coefficients in $q_S$:
\[
\dots, 0, 0;\ \ 1, 1, \dots, 1, 1;\ \ 0, 0, \dots
\]
and this is clearly log-concave. Thus $q_S$ is also ADLC.

However, one can compute the following monomials and their coefficients
in $p \cdot q_S$:
\[
x_1^k x_0^2 \mapsto b^2 + b + 1; \qquad
x_1^k x_0 x_2 \mapsto 3b+1; \qquad
x_1^k x_2^2 \mapsto b+2,
\]
and now we compute using that $b \geq 13/2$:
\[
(b+2)(b^2 + b + 1) - (3b+1)^2 = b \cdot b \cdot (b-6) - 3b + 1 \geq b
\cdot \frac{13}{2} \cdot \frac{1}{2} - 3b + 1 \geq \frac{b}{4} + 1 > 0.
\]
Hence $p \cdot q_S$ is not ADLC.
\end{proof}

The above example and result motivate one to ask just how strong (or
weak) hypotheses are required to preserve the ADLC or related properties
for homogeneous polynomials, with or without the weakening~(b) above. We
begin with three classical, interrelated, positive, ``univariate''
results.
The notion of log-concavity is also known in the theory of total
positivity as the $TN_2$ property (``\textit{totally nonnegative of order
2}''). Namely, given a real sequence $(c_n)_{n \in \mathbb{Z}}$, define
the semi-infinite Toeplitz matrix $T_{\bf c} = (a_{i,j})_{i,j \geq 0}$
where $a_{ij} := c_{i-j}$ for all $i,j$. Then ${\bf c}$ or $T_{\bf c}$ is
said to be $TN_r$ for an integer $r \in [1,\infty]$ if all finite
submatrices of $T_{\bf c}$ of size at most $r \times r$ have nonnegative
determinant. The Cauchy--Binet formula gives that (semi-infinite
Toeplitz) $TN_r$ matrices are closed under multiplication.

Now let ${\bf c}$ be a finite positive sequence with no internal zeros,
padded by zeros:
\[
\dots, 0, 0;\ \ c_0, \dots, c_k;\ \ 0, 0, \dots
\]
with all $c_i > 0$ -- and let ${\bf d} = (d_0, \dots, d_l) \in
(0,\infty)^{l+1}$ be another. One can encode these by their generating
functions/polynomials $\Psi_{\bf c}(x) := c_0 + \cdots + c_k x^k$, and
similarly $\Psi_{\bf d}$. Then $T_{\bf c} T_{\bf d}$ corresponds to the
sequence obtained from $\Psi_{\bf c}(x) \Psi_{\bf d}(x)$, i.e.\ the
convolution product of ${\bf c}, {\bf d}$. Now we record the
aforementioned classical results:
\begin{itemize}
\item For $r=1$, the $TN_r$ property is just nonnegativity. Thus, the
Cauchy--Binet formula yields the (trivial) fact that convolving two
positive sequences yields a positive sequence.

\item For $r=2$, the $TN_r$ property is log-concavity. This yields the
classical fact (see e.g.\ \cite[Chapter~8, Theorem~1.2]{Karlin}) that
convolving two log-concave sequences with no internal zeros yields
another such.

\item For $r=\infty$, the $TN_r$ property is equivalent to the
real-rootedness of $\Psi_{\bf c}(x)$, by celebrated 1950s results of
Edrei~\cite{Edrei,Edrei2} and Aissen--Schoenberg--Whitney~\cite{ASW} --
and ${\bf c}$ is then termed a (finite) \textit{P\'olya frequency
sequence}. Translating modulo this deep equivalence, the convolution fact
is again trivial: the product of two real-rooted polynomials is
real-rooted.
\end{itemize}

The $r=2$ fact was ``upgraded'' in two ways by Br\"and\'en--Huh. The
first is \cite[Corollary 3.8]{BH}: denormalized Lorentzian polynomials
$p$ (i.e.\ $N(p)$ is Lorentzian) are closed under multiplication.
In another direction, the $r=2$ fact was first extended by
Liggett~\cite[Theorem~2]{Lig97} to the univariate statement that the
convolution of two ultra log-concave sequences with no internal zeros is
another such.
In turn, this was extended to the multivariate result
\cite[Corollary~2.32]{BH}, which moreover answers a question of Gurvits
(1990) by showing that the product of strongly log-concave homogeneous
(multivariate) polynomials is strongly log-concave.

Given this multitude of positive results, it is natural to ask if there
is a ``naive'' multivariate generalization of the $r=2$ fact. The
multivariate generalization of log-concavity is simply the ADLC property
(after first homogenizing). However, as the following counterexample
shows, preservation of ADLC under products fails even if one polynomial
is in 2 variables (or homogenized to 3 variables) and the other is a
univariate polynomial -- whose coefficients can even be taken to be
(ultra) log-concave. We write down the result for homogeneous
polynomials; the interested reader may reduce one variable in each by
dehomogenizing.

\begin{proposition}
Fix positive real scalars $a,b>0$ and define the family
\[
p_{b,t}(x,y,z)
:= b^2 \cdot x^2y^2 + b \cdot x^2yz + x^2z^2 + b^2 \cdot xy^2z + b^2
\cdot y^2z^2 + t \cdot xyz^2, \qquad t > 0.
\]
Then for any $t > 2b+a$ and any homogeneous polynomial $q(x,y) = x^k +
a x^{k-1}y + \cdots$ with nonnegative coefficient on $x^{k-2}y^2$,
$p_{b,t}$ is ADLC but $p_{b,t}q$ is not.
\end{proposition}

Note that $p_{b,t}$ has M-convex support and is ADLC:

\begin{center}
	\begin{tabular}{c c c}
		1 & $t$ & $b^2$ \\
		$b$ & $b^2$ & \\
		$b^2$ & &
	\end{tabular}
\end{center}

\noindent Moreover, one can choose $q$ to have all nonnegative
coefficients, even ones forming an ultra log-concave (hence ADLC)
sequence. And yet, $p_{b,t}q$ is not ADLC when $t > 2b+a$.

\begin{proof}
Let $c \geq 0$ be the coefficient of $x^{k-2} y^2$. Now compute the
coefficients of the following monomials in $p_{b,t} q$:
\[
x^{k+2} y^2   \mapsto b^2; \qquad
x^{k+1} y^2 z \mapsto b^2 + ab; \qquad
x^k y^2 z^2   \mapsto b^2 + at + c.
\]
Therefore $p_{b,t}q$ is not ADLC, since
\[
(b^2 + ab)^2 - b^2(b^2+at+c) = b^2 (2ba + a^2 - ta - c) \leq b^2a (2b+a -
t) < 0. \qedhere
\]
\end{proof}
%}}}

%{{{1 Section 5 - Three instances of failure of discrete log-concavity
\section{Three instances of failure of discrete log-concavity}\label{Shigherorder}

We now explain how our log-concavity
result is tight in a precise sense coming from representations of Lie
algebras and symmetric functions.

\subsection{Failure outside type $A$}\label{SG2}

Recall that we have unified log-concavity results for two prominent
families of highest weight representations over
$\mathfrak{sl}_{n+1}(\mathbb{C})$:
(a)~finite-dimensional simple modules; and
(b)~Verma modules.
It is natural to ask if these phenomena hold over general complex
semisimple Lie algebras.

Unfortunately, this was already disproved in~\cite[Figure 2]{HMMS} for the
family~(a) over $C_2 = \mathfrak{sp}_4(\mathbb{C})$. Thus, one can ask if
it also fails for the family~(b) over some other (semi)simple Lie
algebra. The following example shows that it does.

\begin{example}\label{ExG2}
(The three weights in this example $\mu, \mu \pm (\alpha+\beta)$ were
located using SageMath by G.V.\ Krishna Teja and K.\ Hariram, who then
communicated them to us.)
Let $\mathfrak{g}$ be the complex simple Lie algebra of type $G_2$. We
claim that the Kostant partition function is not discretely log-concave
at $n(\alpha + \beta)$ for $n=4,5,6$, where $\alpha, \beta$ are the two
simple roots.

To see why, first assume that  $\alpha$ is short, and order the positive
roots as
\[
\alpha, \quad \beta, \quad \alpha + \beta, \quad 2\alpha + \beta, \quad 3
\alpha + \beta, \quad 3 \alpha + 2 \beta.
\]

Next, define for ${\bf n} = (n_1, \dots, n_6) \in \mathbb{Z}_{\geq 0}^6$
the map $\Psi \colon \mathbb{Z}_{\geq 0}^6 \to \Lambda$, sending
\[
{\bf n} \quad \mapsto \quad n_1 \alpha + n_2 \beta + n_3 (\alpha+\beta) +
n_4 (2\alpha + \beta) + n_5 (3\alpha + \beta) + n_6 (3 \alpha + 2 \beta).
\]

We now list three weights in arithmetic progression, whose multiplicities
are not log-concave: $n\alpha+n\beta$ for $n=4,5,6$. First,
\begin{align*}
\Psi^{-1}(4\alpha + 4\beta) = &\
\{ (a,a,4-a,0,0,0) : 0 \leq a \leq 4 \} \sqcup 
\{ (b,b+1,2-b,1,0,0) : 0 \leq b \leq 2 \} \sqcup\\
&\ \sqcup \{ (0,2,0,2,0,0), (1,3,0,0,1,0), (0,2,1,0,1,0), (1,2,0,0,0,1),
(0,1,1,0,0,1) \},
\end{align*}
so $K_0(4\alpha + 4\beta) = |\Psi^{-1}(4\alpha + 4\beta)| = 13$. Next, to
obtain $5\alpha + 5\beta$, one can start by adding either $(\alpha) +
(\beta)$ to these 13 vector partitions, or $(\alpha+\beta)$ -- and then
considering others (which have neither $(\alpha)+(\beta)$ nor
$(\alpha+\beta)$ in their decomposition). Thus,
\begin{align*}
\Psi^{-1}(5\alpha + 5\beta) = &\ \{ {\bf n} + (1,1,0,0,0,0) : \Psi({\bf
n}) = 4\alpha + 4\beta \} \sqcup \{ (0,0,5,0,0,0), (0,1,3,1,0,0) \}
\sqcup\\
&\ \sqcup \{ (0,2,1,2,0,0), (0,2,2,0,1,0), (0,1,2,0,0,1) \} \sqcup \{
(0,3,0,1,1,0), (0,2,0,1,0,1) \},
\end{align*}
so $K_0(5\alpha + 5\beta) = 20$. Finally, via a similar procedure,
\begin{align*}
\Psi^{-1}(6\alpha + 6\beta) = &\ \{ {\bf n} + (1,1,0,0,0,0) : \Psi({\bf
n}) = 5\alpha + 5\beta \} \sqcup \{ (0,0,6,0,0,0), (0,1,4,1,0,0) \}
\sqcup\\
&\ \sqcup \{ (0,2,2,2,0,0), (0,2,3,0,1,0),
(0,1,3,0,0,1),  (0,3,1,1,1,0), (0,2,1,1,0,1) \} \sqcup\\
&\ \sqcup \{ (0,3,0,3,0,0), (0,4,0,0,2,0), (0,3,0,0,1,1), (0,2,0,0,0,2) \},
\end{align*}
and so $K_0(6\alpha + 6\beta) = 31$. But $20^2 = 400 < 403 = 13 \cdot
31$, so log-concavity of the Kostant partition function fails in type
$G_2$. \qed
\end{example}

\subsection{Log-concavity fails for Jack and Macdonald polynomials}

Given that neither Verma nor finite-dimensional modules have characters
with log-concave coefficients across all simple Lie types, we return to
-- and henceforth work in -- type $A$. In this subsection, we examine if
the log-concavity of the Kostka numbers in $s_\lambda$ (shown
in~\cite[Theorem 2]{HMMS}) extends to certain important generalizations of these
symmetric functions, which are the subject of tremendous recent and
ongoing research.

The first such family is that of \textit{Hall--Littlewood polynomials},
introduced by Hall in the 1950s (see~\cite{Hall}) and by
Littlewood~\cite{Littlewood}. These are a one-parameter generalization of
Schur polynomials, which connect to enumerating subgroups of finite
abelian $p$-groups, to $GL_n$-representations over finite and $p$-adic
fields, and to canonical bases over quantum groups. In the simplest case
of $n=2$, given a partition $\lambda = (a,b)$ (so $a \geq b \geq 0$), we
have:
\begin{equation}\label{EHall}
P_{(a,b)}(x,y; t) := \begin{cases}
s_{(a,b)}(x,y) = x^a y^b + x^{a-1} y^{b+1} + \cdots + x^b y^a, \qquad &
\text{if } a \leq b+1,\\
s_{(a,b)}(x,y) - t s_{(a-1,b+1)}(x,y), & \text{if } a > b+1.
\end{cases}
\end{equation}
Thus for $t=0$ we get the Schur polynomial, while $t=1$ yields the
monomial symmetric function.

In a separate direction, Jack~\cite{Jack} proposed a common
generalization of zonal polynomials (which relate to multivariate
statistics and to $GL_n(\mathbb{R})$-representations) and Schur
polynomials. These are called \textit{Jack polynomials}, and they possess
a ``Jack parameter'' $\tau$. Again when $n=2$, we have:
\begin{align}\label{EJack}
\begin{aligned}
P_{(a,b)}^{(\tau)}(x,y) := &\ \frac{(a-b)!}{(\tau)_{a-b}}
\sum_{i=0}^{a-b} \frac{(\tau)_i (\tau)_{a-b-i}}{i! (a-b-i)!} x^{b+i}
y^{a-i},\\
\text{where} \quad (\tau)_k := &\ \tau(\tau+1)\cdots(\tau+k-1) \text{ for
} k>0; \ (\tau)_0 := 1.
\end{aligned}
\end{align}
Thus we recover Schur polynomials when $\tau=1$, and there are other
well-known specializations among $\tau \in [0,\infty]$. See
e.g.~\cite{Macdonald} for more on both Hall--Littlewood as well as Jack
polynomials.

In fact, these two families have a common, overarching generalization:
\textit{Macdonald polynomials}. Once again if $n=2$ and $\lambda =
(a,b)$, using the combinatorial formula in
\cite[VI.~(7.13$'$)]{Macdonald} we have:

\begin{align}\label{EMacdonald}
\begin{aligned}
P_{(a,b)}(x,y; q,t) = P_{(a,b)}(x,y; 1/q, 1/t) := &\ \sum_{i=0}^{a-b}
\frac{\qpoch{q}{q}{a-b}}{\qpoch{q}{q}{i}\qpoch{q}{q}{a-b-i}}
\frac{\qpoch{t}{q}{i}\qpoch{t}{q}{a-b-i}}{\qpoch{t}{q}{a-b}} x^{b+i} y^{a-i}\\
= &\ \frac{\qpoch{q}{q}{a-b}}{\qpoch{t}{q}{a-b}} \sum_{i=0}^{a-b}
\frac{\qpoch{t}{q}{i}}{\qpoch{q}{q}{i}}
\frac{\qpoch{t}{q}{a-b-i}}{\qpoch{q}{q}{a-b-i}} x^{b+i} y^{a-i},
\end{aligned}
\end{align}
where $\displaystyle \qpoch{z}{q}{k} := \prod_{i=0}^{k-1} (1-zq^i)$
is the $q$-Pochhammer symbol for $k>0$, and $\qpoch{z}{q}{0} := 1$.
In all of the above cases, note that the monomials $x^a y^b, x^b y^a$
have coefficient $1$.

One now checks that setting
(a)~$q=t$, (b)~$q=0$, (c)~$t=1$ in $P_{(a,b)}(x,y;q,t)$ recovers
(a)~the Schur polynomial $s_{(a,b)}(x,y)$,
(b)~the Hall--Littlewood polynomial $P_{(a,b)}(x,y;t)$,
(c)~the monomial symmetric function~$m_{(a,b)}(x,y)$, respectively;
and (d)~if one sets $t = q^\tau$ and lets $q \to 1$ then this
recovers the Jack polynomial $P_{(a,b)}^{(\tau)}$ too. Again
see~\cite{Macdonald} for details.

Additionally, an interesting property satisfied by all of the above
families is \textit{stability}. A consequence of this is: specializing
any of these polynomials $P_{(a,b)}$ in $n>2$ variables $x_i$, to $x_n =
0$, yields the same polynomial in $n-1$ variables. Thus, the coefficients
of the various monomials stay unchanged upon specializing to two
variables (i.e.\ setting $x_3 = \cdots = x_n = 0$).

Given this, we now show via examples that for all three families of
symmetric polynomials $P_{(a,b)}$ which moreover have nonnegative real
coefficients, the coefficients are not log-concave (along the type $A_1$
root direction $y/x$) in general.

\begin{example}
	For Macdonald polynomials, it is customary to consider either $q,t \in
	(0,1)$ or $q,t \in (1,\infty)$ (by~\eqref{EMacdonald}) -- in which case
	all monomials $x^i y^j$ have nonnegative coefficients. Let us take $q,t
	\in (0,1)$. Now we explicitly see that
	\begin{equation}\label{Emac20}
		P_{(2,0)}(x,y; q,t) = x^2 + \frac{(1-t)(1+q)}{1-tq} xy + y^2,
	\end{equation}
	so the coefficients are log-concave if and only if $(q-t)(q-t + 2-2qt) \geq 0$. For each $t>0$, this holds if $q \geq t$ in $(0,1)$ but fails as
	$q \to 0^+$. Thus, log-concavity is not ``universally'' true in the
	region $(0,1)^2$.	
%	so the coefficients are log-concave if and only if 
%	\begin{align*}
%		\frac{(1-t)^2(1+q)^2}{(1-tq)^2}\geq 1,
%		\quad\mbox{or equivalently}\quad
%		\frac{(1+q)^2}{(1-tq)^2}\geq \frac{1}{(1-t)^2}.
%	\end{align*}
%	%$q-t + tq - (tq)^2\geq 0$. 
%	Note that for each $t>0$
%	\[\lim_{q\to 0^+} \frac{(1+q)^2}{(1-tq)^2} = 1< \frac{1}{(1-t)^2}. \]
%	On the other hand, when $q\geq t>0$
%	\[\frac{(1+q)^2}{(1-tq)^2} \geq \frac{1}{(1-tq)^2} \geq \frac{1}{(1-t)^2}.\]
%	Thus, log-concavity is not ``universally'' true in the
%	regime $(0,1)^2$.
	
	As a special case, taking $q \to 0^+$ yields the Hall--Littlewood
	polynomial
	$P_{(2,0)}(x,y;t) = x^2 + (1-t) xy + y^2$,
	and its coefficients too are not log-concave for $t \in (0,1)$. \qed
\end{example}

\begin{example}
We show that the other specialization of Macdonald polynomials above --
namely, Jack polynomials -- also do not have log-concave coefficients in
general.
Let $n=2$, $a > n$, and  $b=0$. Write $P_{(a,0)}^{(\tau)}(x,y) =
\frac{a!}{(\tau)_a} \sum_{i=0}^a c_i x^i y^{a-i}$. We check if $c_1^2
\geq c_0 c_2$ for all $\tau \in [0,\infty]$:
\[
\frac{c_1^2}{c_0 c_2} = \frac{\tau (\tau + a - 2) \cdot
2a}{(\tau+1)(\tau+a-1) \cdot (a-1)}.
\]
This ratio exceeds $1$ if and only if
\[
0 \leq \tau (\tau + a - 2) \cdot 2a - (\tau+1)(\tau+a-1) \cdot (a-1) =
(\tau - 1) (a^2-2a-1 + \tau(a+1)).
\]
As $a \geq 3$, the second factor is positive. Thus log-concavity holds if
and only if $\tau \in [1,\infty]$, whence not for all $\tau$. \qed
\end{example}

For completeness, we also record here another related (but disparate)
notion of log-concavity for Schur polynomials, that is once again
``tight'' among these more general symmetric functions. This is along
\textit{partitions}. Namely, Okounkov had shown~\cite{Okounkov} that
given partitions $\lambda, \mu, \nu$, one has monomial log-concavity:
\[
\lambda + \mu = 2 \nu \quad \implies \quad s_\nu^2(x_1, x_2, \dots) -
s_\lambda(x_1, x_2, \dots) s_\mu(x_1, x_2, \dots) \in \mathbb{Z}_{\geq
0}[x_1, x_2, \dots].
\]

In turn, this implies numerical log-concavity at every point in the
positive orthant $(0,\infty)^n$ (setting $x_{n+1} = x_{n+2} = \cdots = 0
< x_1, \dots, x_n$). Moreover, this monomial log-concavity can be further
upgraded to Schur log-concavity, as conjectured by Okounkov and shown by
Lam--Postnikov--Pylyavskyy~\cite{LPP}.

Given the above results and counterexamples, one can ask if Okounkov's
conjecture -- or even its most basic consequence of \textit{numerical}
log-concavity -- holds for Hall--Littlewood, Jack, or Macdonald
polynomials. We record via some basic examples that it does not, even for
$n=2$.

\begin{example}
As above, we first consider Macdonald polynomials. Using~\eqref{Emac20}
and that
\[
P_{(3,0)}(x,y;q,t) = x^3 + y^3 + \frac{(1-t)(1+q+q^2)}{1-tq^2} xy(x+y),
\]
we compute:
\begin{align*}
&\ P_{(2,0)}(x,y;q,t)^2 - P_{(3,0)}(x,y;q,t) P_{(1,0)}(x,y;q,t)\\
= &\ \frac{(q-t)(1-q)(1+tq)}{(1-tq)(1-tq^2)} xy (x^2+y^2) + \left[ 2 +
\frac{(1-t)^2 (1-q)^2}{(1-tq)^2} - \frac{2(1-t)(1+q+q^2)}{(1-tq^2)}
\right] x^2y^2.
\end{align*}

As $q \to 0^+$ (or at $q=0$), we get
\[
-txy (x^2 + y^2) + (1+t^2) x^2y^2 = -xy(tx-y)(x-ty),
\]
and for any $t \in (0,1)$, this expression is negative for $0 < y < tx <
x$. (This also covers the case of Hall--Littlewood polynomials.) Hence it
remains negative for small $q>0$, and is also not monomial positive
either.
\end{example}

\begin{example}
For Jack polynomials, $P_{(1,0)}^{(\tau)}(x,y) = x+y$,
$P_{(2,0)}^{(\tau)}(x,y) = x^2 + y^2 + \frac{2\tau}{\tau+1} xy$, and
$P_{(3,0)}^{(\tau)}(x,y) = x^3 + y^3 + \frac{3\tau}{\tau+2} xy(x+y)$.
Hence,
\begin{align*}
S_\tau(x,y) := &\ P_{(2,0)}^{(\tau)}(x,y)^2 - P_{(3,0)}^{(\tau)}(x,y)
P_{(1,0)}^{(\tau)}(x,y)\\
= &\ \frac{2(\tau-1)}{(\tau+1)(\tau+2)} xy(x^2+y^2) + \frac{4(\tau^2 +
\tau + 1)}{(\tau+1)^2 (\tau+2)} x^2y^2.
\end{align*}
Let $\tau \in [0,1)$; then this is not monomial positive. Moreover, let
$S'(x,y) := S_\tau(x,y) / (xy)$; then $S'(x,0) < 0$ for all $x>0$. Hence
by continuity, $S'(x,y) < 0$ for $x>0$ and small $y>0$, whence
$S_\tau(x,y) < 0$ as well. \qed
\end{example}

\subsection{Failure for higher-order Verma modules in type $A$}

Recall that Verma modules and parabolic
Vermas (e.g.\ finite-dimensional simple modules) are examples of modules
with (a)~a universal highest weight property, and
(b)~a Weyl-type character formula, arising from
(c)~ a BGG-type resolution via direct sums of Verma modules.
In fact these are part of a bigger family of highest weight
$\mathfrak{g}$-modules (not merely over $\mathfrak{sl}_{n+1}(\mathbb{C})$
but over any Kac--Moody Lie algebra) which satisfy~(a) and (proved in
some cases)~(b) and~(c). These modules were uncovered in recent
work~\cite{Teja}, where they were termed ``higher order Verma modules''.

There is a fourth notable feature of these modules: (d)~Parabolic Verma
modules not only have Weyl-type character formulas, but they also yield
the weight-sets of all simple highest weight modules (including the
non-integrable ones) -- not just in finite type \cite{Kh1} but over all
Kac--Moody Lie algebras \cite{DK}. Similarly, higher order Verma modules
yield the weight-sets of \textit{all} highest weight modules, again over
arbitrary Kac--Moody $\mathfrak{g}$ \cite{Teja}. Thus, they are a natural
family to study beyond parabolic Verma modules; in particular, here we
explore the question of log-concavity of their characters.

\subsection{Preliminaries on higher order Verma modules}

We first introduce the key notion needed to define higher order Verma
modules. A \textit{hole} is defined \cite{Teja} to be an independent
(i.e.\ pairwise orthogonal) set $H \subseteq I$ of simple roots/nodes
in the Dynkin diagram of $\mathfrak{g}$. Given a hole $H \subseteq I$ and
a highest weight $\lambda \in \Lambda^+_H$ (see~\eqref{ELambdaJplus}),
the corresponding higher order Verma module is
\begin{equation}\label{EHOVerma}
\mathbb{M}(\lambda, \{ H \}) := M(\lambda) / U \mathfrak{g} \cdot
\prod_{i \in H} f_i^{\lambda(h_i)+1} \cdot m_\lambda.
\end{equation}
Note that the denominator is a submodule of $M(\lambda)$ that is
isomorphic to the Verma module $M(\prod_{i \in H} s_i \bullet \lambda)$;
and the $f_i, i \in H$ pairwise commute, as do the $s_i$. Moreover, this
quotient module obviously has a Weyl-type character formula, in fact a
2-step resolution by ``usual'' Verma modules:
\begin{align*}
& 0 \to M(\prod_{i \in H} s_i \bullet \lambda) \to M(\lambda) \to
\mathbb{M}(\lambda, \{ H \}) \to 0;\\
& \ch \mathbb{M}(\lambda, \{ H \}) = \sum_{w \in W_{\mathcal{H}} }
(-1)^{\ell_\mathcal{H}(w)} \ch M(w \bullet \lambda),
\end{align*}
where $W_{\mathcal{H}} = \{ e, w_\circ := \prod_{i \in H} s_i \} \cong
\mathbb{Z} / 2 \mathbb{Z}$ and the associated length function is
$\ell_\mathcal{H}(e) = 0$, $\ell_\mathcal{H}(w_\circ) = 1$.

In general, a higher order Verma module involves quotienting $M(\lambda)$
by $U \mathfrak{g} \cdot \prod_{i \in H} f_i^{\lambda(h_i)+1} \cdot
m_\lambda$ for multiple holes $H$. (There can only be finitely many such,
since each $H \subseteq I$.)
For example, if each hole is a singleton $\{ i \}$, and the set of these
is $J$, then (a)~necessarily $\lambda \in \Lambda^+_J$, and (b)~we obtain
precisely the parabolic Verma module $M(\lambda, J)$~\eqref{Egvm}.
More generally, we have:

\begin{definition}\label{DHOVerma}
Let $\mathcal{H} = \{ H_1, \dots, H_l \}$ be a collection of holes --
i.e.\ each $H_j \in {\rm Indep}(I)$. Given a weight $\lambda \in
\bigcap_{j=1}^l \Lambda_{H_j}^+$, the corresponding \textit{higher order
Verma module} is
\[
\mathbb{M}(\lambda, \mathcal{H}) := \frac{M(\lambda)}{\sum_{j=1}^l U
\mathfrak{g} \cdot \prod_{i \in H_j} f_i^{\lambda(h_i)+1} \cdot
m_\lambda}.
\]
\end{definition}

We also need the notion of \textit{minimal holes}. For example if
$\lambda = 0$ and $\mathfrak{g} = \mathfrak{sl}_6(\mathbb{C})$, then $f_1
\overline{m_0} = 0$ in $M(\lambda, \{ 1 \})$, which automatically implies
$f_i f_1 \overline{m_0} = 0$ for all $i>2$. Thus for example,
\[
\mathbb{M}(0, \{ \{ 1 \} \}) = 
\mathbb{M}(0, \{ \{ 1 \}, \{ 1, 3 \}, \{ 1, 4 \}, \{1,5\},\{ 1, 3, 5 \} \}).
\]

Thus, henceforth we will always replace $\mathcal{H}$ by the subset of
``minimal holes'' $\mathcal{H}^{\min}$. Notice that this consists of
irredundant holes $H$.

\begin{definition}
Given $\mathcal{H} \subseteq 2^I$ and $\lambda$ as in
Definition~\ref{DHOVerma}, the module $\mathbb{M}(\lambda, \mathcal{H}) =
\mathbb{M}(\lambda, \mathcal{H}^{\min})$ is said to be an {\em $m$th
order Verma module}, where $m = \max_{H \in \mathcal{H}^{\min}} |H|$.
\end{definition}

Thus, parabolic Verma modules are first order:
\[
M(\lambda, J) = \mathbb{M}(\lambda, \{ \{ i \} : i \in J \}),
\]
while by convention we say that the ``usual'' Verma module $M(\lambda) =
\mathbb{M}(\lambda, \emptyset)$ is zeroth order (as is $0 =
\mathbb{M}(\lambda, \{ \emptyset \})$). The module $\mathbb{M}(\lambda, \{
H \})$ in~\eqref{EHOVerma} is $|H|$th order.

\begin{remark}
For there to exist an $m$th order Verma module over
$\mathfrak{sl}_{n+1}(\mathbb{C})$, it is necessary for an independent
subset of size $m$ to exist within the Dynkin diagram on $I = [n]$. Thus
$n \geq 2m-1$. In particular, there are no second (or higher) order Verma
modules over $\mathfrak{sl}_2(\mathbb{C})$ or
$\mathfrak{sl}_3(\mathbb{C})$ -- one only has Vermas and parabolic
Vermas.
\end{remark}

\subsection{The negative result}

We now come to the goal of this section: showing that over
$\mathfrak{sl}_{n+1}(\mathbb{C})$, higher order Verma characters are not
log-concave along type $A$ root directions. We begin by writing out the
simplest example, before proceeding to the general result.

\begin{example}\label{E2orderVerma}
Let $\mathfrak{g} = \mathfrak{sl}_4(\mathbb{C})$, and let
\[
\lambda = 0, \qquad V = \frac{M(0)}{U \mathfrak{g} \cdot f_1 f_3 \cdot
m_0} = \frac{M(0)}{M(-\alpha_1-\alpha_3)} = \mathbb{M}(0, \{ \{ 1, 3 \}
\}).
\]
This is a second order Verma module. Let $\beta = \alpha_3$ and consider
the $\beta$-root string $\{ - \alpha_1 - \alpha_2 - p \alpha_3 : p =
1,2,3 \}$.
The respective weight spaces of the two Verma modules whose quotient is
$V$ are listed in Table~\ref{Ttable}, via monomials in the ordered PBW
basis whose roots are the following ordered sequence of positive roots in
$\mathfrak{n}^+$:
\[
\alpha_1, \quad \alpha_2, \quad \alpha_3, \quad \alpha_1 + \alpha_2,
\quad \alpha_2 + \alpha_3, \quad \alpha_1 + \alpha_2 + \alpha_3.
\]

\begin{table}[ht]
\begin{tabular}{| c || c | c | c |}
\hline
$\mu$ & Basis of $M(0)_\mu$ & \quad Basis of
$M(-\alpha_1-\alpha_3)_\mu$ \hspace*{2mm} & \quad $\dim V_\mu$
\hspace*{2mm} \\
\hline
\hline
$-\alpha_1-\alpha_2-\alpha_3$
& $f_{\alpha_1} f_{\alpha_2} f_{\alpha_3}$,\ \ $f_{\alpha_3}
f_{\alpha_1+\alpha_2}$,
& $f_{\alpha_2}$
& 3\\
& $f_{\alpha_1} f_{\alpha_2+\alpha_3}$,\ \
$f_{\alpha_1+\alpha_2+\alpha_3}$ & & \\
\hline
$-\alpha_1-\alpha_2-2\alpha_3$
& $f_{\alpha_1} f_{\alpha_2} f_{\alpha_3}^2$,\ \ $f_{\alpha_3}^2
f_{\alpha_1+\alpha_2}$, \hspace*{2mm}
& $f_{\alpha_2} f_{\alpha_3}$,\ \ $f_{\alpha_2+\alpha_3}$
& 2\\
& \quad $f_{\alpha_1} f_{\alpha_3} f_{\alpha_2+\alpha_3}$,\ \
$f_{\alpha_3} f_{\alpha_1+\alpha_2+\alpha_3}$ \hspace*{2mm} & & \\
\hline
$-\alpha_1-\alpha_2-3\alpha_3$
& $f_{\alpha_1} f_{\alpha_2} f_{\alpha_3}^3$,\ \ $f_{\alpha_3}^3
f_{\alpha_1+\alpha_2}$, \hspace*{2mm}
& $f_{\alpha_2} f_{\alpha_3}^2$,\ \ $f_{\alpha_3} f_{\alpha_2+\alpha_3}$
& 2\\
& \quad $f_{\alpha_1} f_{\alpha_3}^2 f_{\alpha_2+\alpha_3}$,\ \
$f_{\alpha_3}^2 f_{\alpha_1+\alpha_2+\alpha_3}$ \hspace*{2mm} & & \\
\hline
\end{tabular}
\vspace{1ex}
\caption{}	
\label{Ttable}
\end{table}

From the table it is clear that $(\dim V_\mu)^2 < \dim V_{\mu+\beta} \dim
V_{\mu-\beta}$ for $\mu = -\alpha_1-\alpha_2-2\alpha_3$ and $\beta =
\alpha_3$. This violates log-concavity of the character of this
second order Verma module $V = \mathbb{M}(0, \{ \{ 1, 3 \} \})$. \qed
\end{example}

Example~\ref{E2orderVerma} is prototypical of the general situation:
the characters of the $m$th order Verma modules~\eqref{EHOVerma} are
never log-concave for $m \geq 2$. More strongly, we have the following result.

\begin{theorem}\label{Thovm}
Fix $\mathfrak{g} = \mathfrak{sl}_{n+1}(\mathbb{C})$ as usual.
Given any set of holes $\mathcal{H} = \{ H_1, \dots, H_l \}$, each of
which has size at least $2$, and a weight $\lambda \in \bigcap_{j=1}^l
\Lambda^+_{H_j}$, the character of the higher order Verma module
$\mathbb{M}(\lambda, \mathcal{H})$ is not log-concave along at least one
type $A$ simple root direction.
\end{theorem}

\begin{proof}
We first prove the case where $\mathcal{H}$ consists of a single (hence
minimal) hole: $\mathcal{H} = \{ H \}$, where $|H| = m \geq 2$. List $H =
\{ i_1 < \cdots < i_m \} \subset [n]$; the corresponding $m$th order
Verma module (as in~\eqref{EHOVerma}) is
\[
\mathbb{M}(\lambda, \{ H \}) := M(\lambda) / M(\lambda - l_1 \alpha_{i_1}
- \cdots - l_m \alpha_{i_m}), \quad \text{where} \quad l_r :=
\lambda(h_{i_r})+1\ \forall r \in [m].
\]
Denote by $K_\lambda(\cdot), K_H(\cdot)$ the KPFs of the Verma modules in
the numerator and denominator, respectively. We now show that their
difference is not log-concave along the $\alpha_{i_2}$-direction; the
proof can be adapted to proceed along the $\alpha_{i_r}$ direction for
any $r \in [m]$.

Set $\beta = \alpha_{i_2}$ and choose $\mu = -\sum_{i=i_1+1}^{i_2-1}
\alpha_i - \beta - \sum_{r=1}^m l_r \alpha_{i_r}$. We will show that
log-concavity fails for the weight multiplicities at $\lambda +
(\mu+\beta), \lambda + \mu, \lambda + (\mu-\beta)$.

To show this, first note that any decomposition of $\mu \pm \beta$ or
$\mu$ as a sum of negative roots involves each $-\alpha_{i_r}$
individually, for $r>2$.
Thus, we obtain the same multiplicities by replacing $\mu$ by $\mu' =
-\sum_{i=i_1+1}^{i_2-1} \alpha_i - \beta - l_1 \alpha_{i_1} - l_2
\alpha_{i_2}$, and $H$ by $H' = \{ i_1, i_2 \}$
-- i.e., replacing the Verma in the denominator by $M(\lambda - l_1
\alpha_{i_1} - l_2 \alpha_{i_2})$.

We now compute the weight space multiplicities of $M(\lambda)$ at
$\lambda + \mu', \lambda + \mu' \pm \beta$ -- in other words, we (replace
$\lambda$ by $0$ and) compute $K_0(\cdot)$ at $-\mu', -\mu' \pm \beta$.
More generally, let $p \in \mathbb{N}$ be arbitrary and consider
\[
-(\mu'+\beta - p \beta) = \sum_{i=i_1+1}^{i_2-1} \alpha_i + l_1
\alpha_{i_1} + (l_2+p) \alpha_{i_2}.
\]
Any decomposition of this into a sum of positive roots would -- akin to
the preceding paragraph -- involve adding $(l_1-1)$ terms $\alpha_{i_1}$
and $(l_2+p-1)$ terms $\alpha_{i_2}$ individually, to
$\sum_{i=i_1}^{i_2} \alpha_i$. Thus $K_0(-(\mu'+\beta - p \beta)) =
K_0(\sum_{i=i_1}^{i_2} \alpha_i)$. But decomposing this sum into positive
type $A$ roots corresponding to a union of contiguous sub-intervals of
$[i_1, i_2]$ involves placing (or not placing) ``barriers/separators'' at
any permissible positions between consecutive entries in $[i_1, i_2]$.
Thus $K_0(\sum_{i=i_1}^{i_2} \alpha_i) = 2^{i_2-i_1}$, which implies from
above that
\[
K_0(-(\mu'+\beta - p \beta)) = 2^{i_2-i_1}, \quad \forall p \in
\mathbb{N}.
\]

We next compute
\[
K_{H'}(-(\mu'+\beta - p \beta)) = K_0 \left( \sum_{i=i_1+1}^{i_2-1}
\alpha_i + p \alpha_{i_2} \right).
\]
Using the same arguments as above, it follows that
\[
K_{H'}(-(\mu'+\beta)) = 2^{i_2-i_1-2}, \qquad
K_{H'}(-(\mu'-p\beta)) = 2^{i_2-i_1-1}\ \text{for}\ p \geq 0.
\]
Putting together these weight multiplicities,
\begin{align}\label{Ectrex}
\begin{aligned}
\dim \mathbb{M}(\lambda, \{ H \})_{\lambda+(\mu+\beta)} = &\
2^{i_2-i_1-2} \cdot 3,\\
\dim \mathbb{M}(\lambda, \{ H \})_{\lambda+\mu} =
\dim \mathbb{M}(\lambda, \{ H \})_{\lambda+(\mu-\beta)} = &\
2^{i_2-i_1-2} \cdot 2.
\end{aligned}
\end{align}
This shows that $\ch \mathbb{M}(\lambda, \{ H \})$ is not log-concave.

We now come to the general case. Enumerate the minimal holes
$\mathcal{H}^{\min} = \{ H_1, \dots, H_l \}$; by assumption, $|H_j| \geq
2\ \forall j$. We choose a hole from $\mathcal{H}^{\min}$ via
the following algorithm:
\begin{enumerate}
\item List the elements of each $H_j$ as $1 \leq i_1^{(j)} < i_2^{(j)} <
\cdots$. Now define $i_1 := \max_{j \in [l]} i_1^{(j)}$ and $J_1 := \{ j
\in [l] : i_1^{(j)} = i_1 \}$.

\item Next, from among these $j$, define $i_2$ to be the smallest ``next
element'', i.e., $i_2 := \min_{j \in J_1} i_2^{(j)}$. Also define $J_2 :=
\{ j \in J_1 : i_2^{(j)} = i_2 \}$.

%\item If there is some $j \in J_2$ for which $H_j = \{ i_1, i_2 \}$ then
%write $H := H_j$ (there is a unique such $j$, by minimality/irredundancy
%of the sets in $\mathcal{H}^{\min}$) and stop.
%
%\item Else, choose $i_3 := \max_{j \in J_2} i_3^{(j)}$ and set
%$J_3 := \{ j \in J_2 : i_3^{(j)} = i_3 \}$. Now if there is some $j \in
%J_3$ for which $H_j = \{ i_1, i_2, i_3 \}$ then
%write $H := H_j$ (again, there is a unique such $j$) and stop.
%
%\item Repeat the preceding step as many times as necessary. (This process
%terminates in finite time because $\mathcal{H}^{\min}$ is finite.)

\item From this set $J_2$, choose any index $j_0$ and fix that minimal
hole $H_{j_0}$.
\end{enumerate}

Now we proceed. As in the special case $\mathcal{H} = \{ H \}$ above, set
$\beta = \alpha_{i_2}$ and $\mu = - \sum_{i=i_1+1}^{i_2-1} \alpha_i -
\beta - \sum_{r=1}^m l_r \alpha_{i_r}$, where $l_r := \lambda(h_{i_r})+1$
for all $r \in [m]$ as above. We show that the log-concavity of $\ch
\mathbb{M}(\lambda, \mathcal{H})$ fails at $\lambda + (\mu + \beta),
\lambda + \mu, \lambda + (\mu - \beta)$.

Given $p \in \mathbb{N}$, define $\mu_p := \mu + \beta - p
\beta$. We claim that the weight space
\begin{equation}\label{Eclaim}
V_{\lambda+\mu_p} = 0 \  \forall p \in \mathbb{N}, \quad \text{where}
\quad V := \sum_{j \in [l], \ j \neq j_0} U \mathfrak{g} \cdot \prod_{i
\in H_j} f_i^{\lambda(h_i)+1} \cdot m_\lambda.
\end{equation}
As $\mathbb{M}(\lambda, \mathcal{H}) \cong \mathbb{M}(\lambda, \{ H_{j_0}
\}) / V$, showing~\eqref{Eclaim} would finish the proof, since it reduces
the computation of weight space dimensions for all $p$ to the previously
considered special case~\eqref{EHOVerma}:
\[
\dim \mathbb{M}(\lambda, \mathcal{H})_{\lambda+\mu_p} = \dim
\mathbb{M}(\lambda, \{ H_{j_0} \})_{\lambda+\mu_p}\ \forall p \geq 0,
\]
and these dimensions were shown above to violate log-concavity for
$p=0,1,2$ in~\eqref{Ectrex}.

We thus conclude by showing~\eqref{Eclaim}. Fix $p \in \mathbb{N}$ and $j
\in [l] \setminus \{ j_0 \}$, and list $H_j = \{ i_1^{(j)} < \cdots <
i_{m'}^{(j)} \}$, where $m' \geq 2$. It suffices to show the sub-claim
that $\dim (V_j)_{\lambda+\mu_p} = 0$, where we set
\[
V_j := U \mathfrak{g} \cdot \prod_{i \in H_j} f_i^{\lambda(h_i)+1} \cdot
m_\lambda \cong M \left( \lambda - \sum_{r=1}^{m'}
(\lambda(h_{i_r^{(j)}})+1) \alpha_{i_r} \right)
\]
for compactness of notation.

To show the sub-claim, list the elements of the hole $H_{j_0}$ as $\{ i_1
< \cdots < i_m \}$ for some $m \geq 2$, and consider two cases for the
index $i_1^{(j)}$ in $H_j$.
If $i_1^{(j)} < i_1$ then all weights of $V_j$ are of the form $\lambda -
\alpha_{i_1^{(j)}} - \sum_{i \in I} a_i \alpha_i$ for $a_i \in
\mathbb{N}$; as the $\alpha_i$ are linearly independent in
$\mathfrak{h}^*$, this would never yield $\lambda+\mu_p$. 

Else by choice of $i_1$ in the algorithm above, $i_1^{(j)} = i_1$, i.e.\
$j \in J_1$. By that same algorithm, now we must have $i_2^{(j)} \geq
i_2$. Hence all $i_r^{(j)} \geq i_2$ for all $r \geq 2$. Now if any
$i_r^{(j)} \not\in H_{j_0}$ then the same weight consideration in the
preceding paragraph shows that $\dim (V_j)_{\lambda+\mu_p} = 0$.

This brings us to the case where all $i_r^{(j)} \in H_{j_0}$. But then
$H_j \subseteq H_{j_0}$, which violates the minimality/irredundancy of
the holes $\mathcal{H}^{\min} = \{ H_1, \dots, H_l \}$. This
contradiction shows that $\dim (V_j)_{\mu_p} = 0$ for $j \neq j_0$, which
in turn shows~\eqref{Eclaim} and completes the proof.
\end{proof}

\begin{remark}
The reason (we suspect) why log-concavity does not go through for higher
order Verma modules is that they cannot be obtained via parabolic
induction. As a prototypical example, let $\mathfrak{g} =
\mathfrak{sl}_2(\mathbb{C}) \oplus \mathfrak{sl}_2(\mathbb{C})$, and
consider the ``simplest'' second order Verma module -- the one with highest
weight $(0,0)$. This is the module
\[
\mathbb{M}((0,0), \{ \{ 1, 2 \} \}) = M(0,0) / M(-2,-2),
\]
and it has zero or one dimensional weight spaces, with weights $-p
\alpha_1, -p \alpha_2$ for $p \in \mathbb{N}$. Already by considering its
character (a sum of two geometric series with ``ratios''
$e^{-\alpha_1}$ and $e^{-\alpha_2}$) we see that this is not a
nontrivial product, hence the module is not induced from a submodule over
a proper Lie subalgebra. This is unlike every parabolic Verma
module over every semisimple Lie algebra, for which the induced
module construction~\eqref{Egvm-tensor} was crucial in proving
log-concavity above.

That said, in this specific instance the character is indeed log-concave
along all root directions; we study this in greater detail in the next
section.
\end{remark}
%}}}

%{{{1 Section 6 - Characters of usual, parabolic, and higher order Vermas over products of type $A$
\section{Characters of usual, parabolic, and higher order Vermas over products of type $A$}\label{Sfinal}

In this concluding section, we generalize our main results in the
previous sections (Theorems~\ref{T1} and~\ref{T4}), going from the family
$\{ \mathfrak{sl}_{n+1} : n \in \mathbb{N} \}$ to a larger family of
complex semisimple Lie algebras. More precisely, we show that
(parabolic) Verma module characters over this larger family are
log-concave, but higher order Verma characters are not.

Fix positive integers $T$ and $n_1, \dots, n_T$, let $\mathfrak{g}_t =
\mathfrak{sl}_{n_t+1}(\mathbb{C})$, and set
\[
\mathfrak{g} = \bigoplus_{t=1}^T \mathfrak{sl}_{n_t+1}(\mathbb{C}) =
\oplus_{t=1}^T \mathfrak{g}_t.
\]
Correspondingly, we set notation: the Dynkin diagram is a disjoint union
of type $A$ connected components, with sets of nodes
\begin{equation}\label{Enodes}
I_t := (t,[n_t]), \qquad
I = \bigsqcup_{t=1}^T I_t = \{ (t,i) : t \in [T], i \in [n_t] \}.
\end{equation}
The set of positive roots is the union of the individual positive
root-sets:
$\Delta = \sqcup_{t=1}^T \Delta_t$,
and similarly for the simple roots. The space of ``highest weights'' is
$\mathfrak{h}^* = \bigoplus_{t=1}^T \mathfrak{h}_t^*$, and given $J
\subseteq I$, $J$ and the space of $J$-dominant integral weights
$\Lambda^+_J \subset \mathfrak{h}^*$ split similarly:
\[
J = \bigsqcup_{t=1}^T J_t, \ J_t = J \cap I_t; \qquad
\Lambda^+_J = \bigoplus_{t=1}^T \Lambda^+_{J_t},
\]
where $\Lambda^+_{J_t} \subset \mathfrak{h}^*_t$. We conclude this work
by showing that the log-concavity of parabolic Verma characters extends
to products of $\mathfrak{sl}_{n+1}$'s, but this again fails for higher
order Vermas (unless $\mathfrak{sl}_2$'s are involved -- in which case
one only has singleton holes in each Dynkin component).

\begin{theorem}\label{Thvm}
{\em (First order case)}
Given $J \subseteq I$, and a highest weight $\lambda =
(\lambda_t)_{t=1}^T \in \Lambda^+_J$, the normalized shifted character
$N(x^\delta \cdot \ch M(\lambda,J))$ of every parabolic Verma module is
Lorentzian, and hence $N(x^\delta \cdot \ch M(\lambda,J))$ is continuously log-concave and $\ch M(\lambda,J)$ is discretely (along all
root directions in $\Delta$) log-concave. Here, $\delta \in \mathbb{N}^d$
is arbitrary, with $d = \sum_{t=1}^T (n_t+1)$.

{\em (Higher order case)}
Next, let $H = \sqcup_{t=1}^T H_t$ be an independent set of simple
roots/nodes in the Dynkin diagram. The following are equivalent for a
weight $\lambda \in \Lambda^+_H$:
\begin{enumerate}
\item The character of the higher order Verma module $\mathbb{M}(\lambda,
\{ H \})$ is discretely log-concave along all root directions in
$\Delta$.

\item $\ch \mathbb{M}(\lambda, \{ H \})$ is discretely log-concave along
all simple root directions.

\item Either $H$ is a singleton set, or for every $t \in [T]$, either
$H_t$ is empty or $H_t$ is a singleton and equal to all of $I_t$  (i.e.,
$n_t = 1$).
\end{enumerate}
\end{theorem}

\begin{proof}
For the first order case, standard results \cite{Hum} yield that (using
the above notation)
\begin{equation}\label{Epvm-tensor}
M(\lambda, J) \cong \bigotimes_{t=1}^T M_{\mathfrak{g}_t}(\lambda_t,
J_t).
\end{equation}
From this it follows -- upon writing $\delta = (\delta_t)_{t=1}^T$ and
decomposing the $d$ variables $x$ into individual $(n_t+1)$-tuples
$x^{(t)}$ -- that
\[
N(x^\delta \cdot \ch M(\lambda,J)) = \prod_{t=1}^T N((x^{(t)})^{\delta_t}
\cdot \ch M_{\mathfrak{g}_t}(\lambda_t, J_t)),
\]
and this is Lorentzian by Theorem~\ref{T1}, hence $N(x^\delta \cdot \ch
M(\lambda,J))$ is continuously log-concave and $\ch M(\lambda,J)$ is
discretely log-concave by Theorem~\ref{thm:BHdoubleLC}.

We now come to the higher order case. Clearly $(1) \implies (2)$. We next
assume~(3) and show~(1). First if $H$ is a singleton set, say $H = \{ i_1
\} \subseteq I_1$ without loss of generality, then by~\eqref{EHOVerma}
and~\eqref{Epvm-tensor},
\[
\mathbb{M}(\lambda, \{ H \}) = M(\lambda, \{ i_1 \}) \cong
M_{\mathfrak{g}_1}(\lambda_1, \{ i_1 \}) \otimes \bigotimes_{t=2}^T
M_{\mathfrak{g}_t}(\lambda_t),
\]
and we obtain~(1) by the previous part.

Otherwise, first if all $H_t$ are empty then $\mathbb{M}(\lambda, \{
H \}) = M(\lambda)$, and we again reduce to the previous part. Else
assume without loss of generality that $H = \{ i_1, \dots, i_{t_0} \}$
for some $t_0 \in [T]$, with $H_t = \{ i_t \} = I_t$ for $t \in
[t_0]$. Thus $\mathfrak{g}_t \cong \mathfrak{sl}_2(\mathbb{C})$ for $t
\in [t_0]$. Then by~\eqref{Ecommute},
\begin{align*}
U \mathfrak{g} \cdot \prod_{i \in H} f_i^{\lambda(h_i)+1} \cdot m_\lambda
= &\ U \mathfrak{g} \cdot \prod_{t=1}^{t_0} f_{i_t}^{\lambda(h_{i_t})+1}
\cdot m_\lambda\\
\cong &\ \bigotimes_{t=1}^{t_0} M_{\mathfrak{g}_t}(\lambda_t -
(\lambda_t(h_{i_t})+1)\alpha_{i_t}) \otimes \bigotimes_{t=t_0+1}^T
M_{\mathfrak{g}_t}(\lambda_t).
\end{align*}
It follows by setting
\[
\mathfrak{g}' := \oplus_{t=1}^{t_0} \mathfrak{g}_t \cong
\mathfrak{sl}_2(\mathbb{C})^{\oplus t_0},
\qquad \lambda' := \oplus_{t=1}^{t_0} \lambda_t
\]
that
\[
\mathbb{M}(\lambda, \{ H \}) \cong \mathbb{M}_{\mathfrak{g}'}(\lambda',
\{ H \}) \otimes \bigotimes_{t=t_0+1}^T M_{\mathfrak{g}_t}(\lambda_t).
\]

As Verma module characters (i.e.\ KPFs) are log-concave \cite{HMMS},
and the characters of the tensor factors here are in disjoint sets
of variables, to deduce~(1) it suffices to show that $\ch
\mathbb{M}_{\mathfrak{g}'}(\lambda', \{ H \})$ is discrete log-concave
along the (simple) root directions $\alpha_{i_1}, \dots,
\alpha_{i_{t_0}}$. But
\begin{equation}\label{Edefn}
\mathbb{M}_{\mathfrak{g}'}(\lambda', \{ H \}) \cong
\frac{\otimes_{t=1}^{t_0}
M_{\mathfrak{g}_t}(\lambda_t)}{\otimes_{t=1}^{t_0}
M_{\mathfrak{g}_t}(\lambda_t - (\lambda_t(h_{i_t})+1)\alpha_{i_t})},
\end{equation}
and all weight spaces in the numerator and denominator are
one-dimensional, by $\mathfrak{sl}_2$-theory. Since the positive/simple
roots in $\mathfrak{g}'$ are pairwise orthogonal, the character of
$\mathbb{M}_{\mathfrak{g}'}(\lambda', \{ H \})$ ``equals'' the
set-difference of ``doubled lattice'' points in shifted negative
orthants:
\[
{\bf v} - 2\mathbb{N}^{t_0} \setminus {\bf w} - 2\mathbb{N}^{t_0},
\quad \text{where} \quad
{\bf v} = (\lambda_t(h_{i_t}))_{t=1}^{t_0}, \
{\bf w} = (-\lambda_t(h_{i_t})-2)_{t=1}^{t_0}.
\]
Now along any ``downward'' ray parallel to a coordinate axis, i.e.\ a
(simple) root direction, the multiplicities in the quotient module either
form a sequence of ones, or read $1, \dots, 1, 0, 0, \dots$. Both
sequences are log-concave, again yielding~(1).

Finally, we show the contrapositive of the implication $(2) \implies
(3)$. There are two cases: first suppose some $H_t$ has size at least 2,
say $H_T$. Set
\[
\lambda' := \lambda - \sum_{t=1}^{T-1} \sum_{i \in H_t}
(\lambda_t(h_i)+1) \alpha_i = {\rm wt} \prod_{i \in H \setminus H_T}
f_i^{\lambda(h_i)+1} \cdot m_\lambda,
\]
and note by ``$\mathfrak{sl}_2^{\oplus (|H|-|H_T|)}$-theory'' that the
KPF-value $\dim M(\lambda)_{\lambda'} = 1$. So for any
$\mathbb{N}$-linear combination of (simple) roots in $\Delta_T$, say
$\gamma \in \mathbb{N} \Delta_T$, it follows that
\begin{equation}\label{Etemp}
\mathbb{M}(\lambda, \{ H \})_{\lambda' - \gamma}
= \bigotimes_{t=1}^{T-1} \left( \mathbb{C} \prod_{i \in H_t}
f_i^{\lambda(h_i)+1} \cdot m_{\lambda_t} \right) \otimes
\mathbb{M}_{\mathfrak{g}_T}(\lambda_T, \{ H_T \})_{\lambda_T - \gamma}.
\end{equation}

By Theorem~\ref{Thovm}, there exist a weight $\mu \in \mathfrak{h}^*_T$
and a simple root $\beta \in \Delta_T$ such that $\mu + \beta \in -
\mathbb{N} \Delta_T$ and the multiplicities $\dim
\mathbb{M}_{\mathfrak{g}_T}(\lambda_T, \{ H_T \})_\beta$ violate
log-concavity at $\lambda_T + \mu, \lambda_T + \mu \pm \beta$. We are now
done by setting $\gamma = -\mu, -\mu \pm \beta$ in~\eqref{Etemp}.

The other case is when all $H_t$ are singletons or empty (which
information we do not use below), at least two $H_t$ are singletons, and
for at least one of these $t$ we have $n_t > 1$. Thus, say 
$H_{T-1} = \{ i_{T-1} \}$ and $H_T = \{ i_T \}$, and $n_T > 1$. This last
yields $i_0 \in I_T$ which is adjacent to $i_T$ in the Dynkin diagram.
Now set
\[
\mu = \lambda - \alpha_{i_0} - \sum_{i \in H} (\lambda(h_i)+1) \alpha_i,
\qquad \beta = \alpha_{i_{T-1}}.
\]

We will show that $\ch \mathbb{M}(\lambda, \{ H \})$ is not log-concave
at the weights $\mu, \mu \pm \beta$. Indeed, since
$\mathbb{M}(\lambda, \{ H \}) \cong M(\lambda) / M(\mu + \alpha_{i_0})$,
we see that
\[
\dim \mathbb{M}(\lambda, \{ H \})_\mu = \dim M(\lambda)_\mu - \dim
M(\mu + \alpha_{i_0})_\mu = \dim M(\lambda)_\mu - 1 = 1,
\]
where (for expositional sake) we detail the proof of the final equality.
The simple roots occurring in $\lambda - \mu$ are $\{ \alpha_i : i \in H
\}$ and $\alpha_{i_0}$. The only connected Dynkin subdiagram in these is
the edge $i_0 \longleftrightarrow i_T$. Thus,
\[
\dim M(\lambda)_\mu = K(\alpha_{i_0} + (\lambda(h_{i_T})+1)
\alpha_{i_T}),
\]
and this equals~2, either by writing this weight as a sum of simple
roots, or as $(\alpha_{i_0} + \alpha_{i_T})$ plus $\lambda(h_{i_T})$-many
copies of $\alpha_{i_T}$. This calculation also applies to show that
$\dim M(\lambda)_{\mu \pm \beta} = 2$. Hence,
\[
\dim \mathbb{M}(\lambda, \{ H \})_{\mu-\beta} = \dim
M(\lambda)_{\mu-\beta} - \dim M(\mu + \alpha_{i_0})_{\mu-\beta} = 2 - 1 =
1.
\]

On the other hand, $\mu + \beta$ is not in the weights of $M(\mu +
\alpha_{i_0}) = \mu + \alpha_{i_0} - \mathbb{N} \Delta$. Thus,
\[
\dim \mathbb{M}(\lambda, \{ H \})_\mu = \dim M(\lambda)_{\mu+\beta},
\]
which equals 2 from above. Summarizing,
\[
\dim \mathbb{M}(\lambda, \{ H \})_{\mu+\beta} = 2, \qquad
\dim \mathbb{M}(\lambda, \{ H \})_\mu =
\dim \mathbb{M}(\lambda, \{ H \})_{\mu-\beta} = 1,
\]
and log-concavity fails along the $\alpha_{i_{T-1}}$-direction.
\end{proof}
%}}}

\subsection*{Acknowledgments}

We thank Petter Br\"and\'en, June Huh, Karola M\'esz\'aros, and Daniel
Qin for valuable discussions. We also thank G.V.\ Krishna
Teja and K.\ Hariram for their SageMath computations that revealed the
counterexample in Section~\ref{SG2}.
A.K.~was partially supported by the SwarnaJayanti Fellowship grants
SB/SJF/2019-20/14 and DST/SJF/MS/2019/3 from SERB and DST (Govt.~of
India), a Shanti Swarup Bhatnagar Award from CSIR (Govt.\ of India), and
the DST FIST program 2021 [TPN--700661]. He is also grateful to North
Carolina State University for their hospitality and excellent working
conditions.
J.M.~received support from NSF Grant DMS-2452179 and Simons Foundation Travel Support for
Mathematicians Award MPS-TSM-00007970.

%{{{1 Bibliography

%	\bibliographystyle{plain}
%	\bibliography{bibliography}

%}}}

\appendix

%{{{1 Appendix A - Conjecture 12 of [HMMS], and its partial resolution
\section{Overview of the question of Lorentzianity and log-concavity of
simple highest weight module characters}\label{Sconj}

Here we explain in some detail, how our results extend and make partial
progress on a conjecture of Huh et al.~\cite{HMMS}. Note that
Theorem~\ref{T1} on the log-concavity of parabolic Verma characters
unifies the results in~\cite{HMMS} for two sub-families:
finite-dimensional simples and Vermas. After proving these two results
the authors stated \cite[Conjecture~12]{HMMS}, which asserts the
Lorentzianity of $N(x^\delta \cdot \ch V)$ for a different family
subsuming these two: all simple highest weight modules $\{ V(\lambda) :
\lambda \in \Lambda \}$\footnote{The precise construction of
$V(\lambda)$ for $\lambda \in \mathfrak{h}^* \setminus P^+$ is not
crucial to this work, but we recall it here for completeness. Every
proper submodule $N$ of $M(\lambda)$ is $\mathfrak{h}$-semisimple, i.e.\
has a basis of $\mathfrak{h}$-eigenvectors. In particular, $N$ cannot
contain the one-dimensional $\lambda$-weight space $\mathbb{C} m_\lambda
\subset M(\lambda)$ as it generates $M(\lambda)$. Thus, neither can the
sum $N_{\max}$ of all proper submodules $N$. Now $N_{\max}$ is the unique
maximum submodule of $M(\lambda)$, and $V(\lambda) = M(\lambda) /
N_{\max}$.} where $\Lambda$ is the integral weight
lattice~\eqref{Eweightlattice}.

We now explain how Theorem~\ref{T1} proves additional cases of this
conjecture (i.e., beyond the cases shown in~\cite{HMMS}). In fact, we do
so in the setting of Theorem~\ref{T5}, which is more general in two ways:
first, Huh et al.\ stated and proved their results only for $T=1$, i.e.\
over $\mathfrak{sl}_{n+1}(\mathbb{C})$; second, they made the conjecture
only for integral weights $\lambda$, whereas Theorem~\ref{T5} and
Conjecture~\ref{Conj} above are for all weights $\lambda \in
\mathfrak{h}^*$. Thus, here is a summary of what follows:
\begin{enumerate}
\item For highest weights $\lambda \in \mathfrak{h}^* \cong \mathbb{C}^I$
that are generic (i.e., avoid countably many hyperplanes, and in
particular form a full-dimensional open set), the simple module
$V(\lambda)$ is in fact the Verma $M(\lambda)$, and hence its
$x^\delta$-shifted character is denormalized Lorentzian by~\cite{HMMS}.
These generic highest weights are called \textit{antidominant}, and we
define and discuss them below.

\item Thus, the interesting phenomena occur for the remaining highest
weights $\lambda \in \mathfrak{h}^*$, with the most special points being
the countable (hence zero-dimensional) set of dominant integral weights
$\lambda \in \Lambda^+$ (see~\eqref{Eweightlattice}). For these
$\lambda$, the $x^\delta$-shifted character of $V(\lambda)$
is again denormalized Lorentzian by~\cite{HMMS}.

\item It is the remaining non-generic $\lambda$ -- lying on countably
many hyperplanes -- for which \cite[Conjecture~12]{HMMS} was proposed.
Our Theorem~\ref{T1} affirmatively resolves the conjecture for subsets of
these $\lambda$ in each ``intermediate'' dimension -- see
Example~\ref{Exdim} below. In particular, we show in
Example~\ref{Exgeneric} the conjecture for all \textit{generic
non-antidominant weights} $\lambda$, i.e.\ ones that lie on a unique
hyperplane mentioned in point~(1) and indexed by a simple root.
We stress that we work with non-integral weights $\lambda$
too, unlike~\cite{HMMS} even in the $T=1$ special case.

\item Since we have extended the conjecture in~\cite{HMMS} to
Conjecture~\ref{Conj}, we in fact explain our partial progress via
Theorem~\ref{T5} in the more general setting of $\oplus_{t=1}^T
\mathfrak{sl}_{n_t+1}(\mathbb{C})$. (To understand our progress in the
original situation of \cite{HMMS}, simply set $T=1$ below.)
This progress is via the well-known Jantzen simplicity criterion
\cite[Satz~4]{Jantzen}, which characterizes the parabolic Verma modules
$M(\lambda, J), \ \lambda \in \Lambda_J^+$ which are simple over
$\mathfrak{g} = \oplus_{t=1}^T \mathfrak{sl}_{n_t+1}(\mathbb{C})$. We
also explain how the conjecture in~\cite{HMMS} implies
Conjecture~\ref{Conj} for integral $\lambda$.
\end{enumerate}

\subsection{Generic Verma modules are simple}

In the remainder of this section we proceed in greater detail, for the
reader who may not be well-versed in representations of semisimple Lie
algebras. First say $\mathfrak{g} = \mathfrak{sl}_{n+1}(\mathbb{C})$,
with root system $\Phi = \{ \alpha_{ij} = \varepsilon_i - \varepsilon_j :
1 \leq i \neq j \leq n+1 \}$. Given a root $\alpha_{ij}$, let
$h_{\alpha_{ij}}$ denote the diagonal matrix $E_{ii} - E_{jj}$. Then
\[
[\mathfrak{g}_{\alpha_{ij}}, \mathfrak{g}_{-\alpha_{ij}}] = [ \mathbb{C}
E_{ij}, \mathbb{C} E_{ji}] = \mathbb{C} h_{\alpha_{ij}}.
\]

More generally, now take $\mathfrak{g} = \oplus_{t=1}^T
\mathfrak{sl}_{n_t+1}(\mathbb{C})$ with $T, n_t \geq 1$. Then
$\mathfrak{h} = \oplus_{t=1}^T \mathfrak{h}_t \cong \oplus_{t=1}^T
\mathbb{C}^{n_t}$, and
\[
\Phi = \sqcup_t \Phi_t = \bigsqcup_{t=1}^T \{ \alpha_{ij}^{(t)} : 1 \leq
i \neq j \leq n_t+1 \},
\]
say. The positive roots are $\Phi^+ = \sqcup_t \Phi_t^+ = \sqcup_t \{
\alpha_{ij}^{(t)} : i < j \}$. Moreover, the set of weights, the weight
lattice, and the dominant integral weights decompose into pairwise
orthogonal subspaces/subsets:
\[
\mathfrak{h}^* = \oplus_{t=1}^T \mathfrak{h}_t^* \quad \supset \quad
\Lambda = \oplus_{t=1}^T \Lambda_t \quad \supset \quad \Lambda^+ =
\oplus_{t=1}^T \Lambda_t^+,
\]
where e.g.\ $\Lambda_t^+ = \{ \lambda_t \in \mathfrak{h}_t^* :
\lambda_t(h_{\alpha_{ij}^{(t)}}) \in \mathbb{N} \}$.

With this notation at hand, we now explain the first point in the summary
above. A weight $\lambda \in \mathfrak{h}^*$ is said to be
\textit{antidominant} if $(\lambda+\rho)(h_\alpha) = 2 (\lambda + \rho,
\alpha) / (\alpha,\alpha)$ is not a positive integer, for every positive
root $\alpha$. Now in the present setting of $\mathfrak{g}$, the Weyl
vector $\rho = (\rho_t)_{t=1}^T \in \Lambda^+$, where
\[
\rho_t = \frac{1}{2} \sum_{i<j} \alpha_{ij}^{(t)} = \frac{1}{2} (n_t,
n_t-2, \dots, 2-n_t, -n_t) \in \Lambda_t^+.
\]
Hence for a weight $\lambda = (\lambda_t)_t$ -- we use $\lambda_t =
(\lambda_1^{(t)}, \dots, \lambda_{n_t+1}^{(t)})$ below -- to be
antidominant means that
\[
(\lambda+\rho)( h_{\alpha_{ij}^{(t)}}) = \frac{2 (\lambda + \rho,
\alpha_{ij}^{(t)})}{(\alpha_{ij}^{(t)}, \alpha_{ij}^{(t)})} = \frac{2
(\lambda_t + \rho_t, \alpha_{ij}^{(t)})}{(\alpha_{ij}^{(t)},
\alpha_{ij}^{(t)})} = \lambda^{(t)}_i - \lambda^{(t)}_j + j-i
\]
is not a positive integer, for every $t \in [T]$ and $i<j$ in $[n_t+1]$.
The set of such $\lambda$ is precisely the complement in $\mathfrak{h}^*$
of a countable collection of hyperplanes, hence generic. Moreover:

\begin{theorem}[{Verma module simplicity criterion,
\cite[Theorem~4.8]{Hum}}]\label{Thum}
The Verma module $M(\lambda)$ is simple if and only if $\lambda \in
\mathfrak{h}^*$ is antidominant.
\end{theorem}

Thus, for $\lambda$ generic (antidominant), the simple highest weight
module $V(\lambda)$ has character equal to the usual Kostant partition
function, shifted by $e^\lambda$. Hence by~\eqref{Ecommute} and the $T=1$
result in \cite{HMMS}, $N(x^\delta \cdot \ch V(\lambda))$ is denormalized
Lorentzian for all antidominant $\lambda \in \mathfrak{h}^*$ and all
$\delta$.\medskip

This explains the first point; the second was discussed at length above.
Next, it is not hard to continue beyond~\eqref{Ecommute} and show that
\[
M(\lambda) \cong \bigotimes_{t=1}^T M_t(\lambda_t) \quad \implies \quad
V(\lambda) \cong \bigotimes_{t=1}^T V_t(\lambda_t)
\]
for all $\lambda = (\lambda_t)_t \in \mathfrak{h}^*$,
where $M_t, V_t$ denote Verma and simple modules over
$\mathfrak{sl}_{n_t+1}(\mathbb{C})$, respectively.
Thus $\ch V(\lambda)$ is the product of $\ch V_t(\lambda_t)$ in disjoint
sets of variables $x_1^{(t)}, \dots, x_{n_t+1}^{(t)}$, and so for
integral weights $\lambda$, Conjecture~\ref{Conj} would follow from
\cite[Conjecture~12]{HMMS} (and Theorem~\ref{thm:BHdoubleLC} (3)).

\subsection{Going beyond \cite{HMMS}: Jantzen's criterion}

Finally, we explain the precise progress that Theorem~\ref{T5} makes
beyond the above, in tackling Conjecture~\ref{Conj} (which subsumes the
third point where $T=1$). Begin with $\lambda = (\lambda_t)_t \in
\mathfrak{h}^*$; then for every $(t,i) \in [T] \times [n_t]$ such that
$\lambda_t(h_{t,i}) \in \mathbb{N}$ (here, $h_{t,i} :=
h_{\alpha_{i,i+1}^{(t)}}$), we have that $U \mathfrak{g} \cdot
f_{t,i}^{\lambda_t(h_{t,i})+1} \cdot m_\lambda$ is a proper submodule of
$M(\lambda)$, and hence its image in the quotient $V(\lambda)$ must
vanish. Setting
\begin{equation}\label{EJt}
J_\lambda := \{ (t,i) \in I : \lambda_t(h_{t,i}) \in \mathbb{N} \}, \qquad
J_{\lambda_t} := J_\lambda \cap (t,[n_t]) = J_\lambda \cap I_t,
\end{equation}
it follows that $J_\lambda$ (respectively, $J_{\lambda_t}$) is the unique
maximum set $J$ of nodes for which $\lambda \in \Lambda_J^+$
(respectively, $\lambda_t \in \Lambda_J^+$); in turn, this implies
$M(\lambda, J_\lambda) \twoheadrightarrow V(\lambda)$. Thus,
\textit{$M(\lambda,J_\lambda)$ is the only parabolic Verma that is
possibly simple.}

We now explicitly write out Jantzen's simplicity criterion for
$M(\lambda, J_\lambda)$ over $\mathfrak{g} = \oplus_{t=1}^T
\mathfrak{sl}_{n_t+1}(\mathbb{C})$. First, let
\begin{equation}\label{Epsi}
\Psi^+_\lambda := \{ \beta \in \Phi^+ \setminus \Phi^+_{J_\lambda} :
(\lambda + \rho)(h_\beta) \text{ is a positive integer} \}.
\end{equation}

In words, $\Psi^+_\lambda$ denotes those roots $\beta =
\alpha_{ij}^{(t)}$ for which
(a)~$i<j$;
(b)~$(t,i), \dots, (t,j-1)$ are not all in $J_{\lambda_t}$; and
(c)~$(\lambda + \rho)(h_\beta)$ is a positive integer.
Now we have:

\begin{theorem}[{Jantzen's simplicity criterion,
\cite[Satz~4]{Jantzen}}]\label{Tjantzen}
Suppose $\mathfrak{g} = \oplus_{t=1}^T \mathfrak{sl}_{n_t+1}(\mathbb{C})$
and $\lambda \in \mathfrak{h}^*$. Then the parabolic Verma module
$M(\lambda, J)$ is simple if and only if $J = J_\lambda$ and ``condition
$(M+)$'' is satisfied:

For all $\beta \in \Psi_\lambda^+$, there is a root $\gamma \in
(\mathbb{Q} \Phi_{J_\lambda} + \mathbb{Q} \beta) \cap \Phi$ such that
$(\lambda + \rho)(h_\gamma) = 0$ and $s_\beta(\gamma) \in
\Phi_{J_\lambda}$.
\end{theorem}

We illustrate the situation via several examples. The first lists three
subcases in which $M(\lambda, J_\lambda)$ is either a Verma or a simple
module, so that the log-concavity of its character follows
from~\cite{HMMS} if $\mathfrak{g} = \mathfrak{sl}_n$, else (if $T>1$
then) from Theorem~\ref{T5} above.

\begin{example}[Examples covered by \cite{HMMS}]
Let $\mathfrak{g} = \oplus_{t=1}^T \mathfrak{sl}_{n_t+1}(\mathbb{C})$, as
above.
\begin{enumerate}
\item Suppose $\lambda \in \Lambda^+$, so that $J_\lambda$ consists of
all Dynkin diagram nodes: $J_\lambda = I = \sqcup_t (t, [n_t])$. Then
$\Psi_\lambda^+$ is empty, so condition $(M+)$ is indeed satisfied;
moreover, the parabolic Verma module is finite-dimensional: $M(\lambda,
J_\lambda) \cong V(\lambda)$.

\item Suppose instead that $\lambda$ is antidominant. Then $J_\lambda$ is
empty, hence so is $\Psi_\lambda^+$ by definition. Thus condition $(M+)$
is again satisfied, and the parabolic Verma is $M(\lambda,\emptyset) =
M(\lambda) = V(\lambda)$ by Theorem~\ref{Thum}.

\item Let $\mathfrak{g} = \mathfrak{sl}_2^{\oplus T}$, so that $\Phi^+ =
\{ \alpha^{(1)}, \dots, \alpha^{(T)} \}$; and let $\lambda \in
\mathfrak{h}^*$ be arbitrary. Then akin to~\eqref{Ecommute}, one has
$M(\lambda, J_\lambda) \cong \otimes_{t=1}^T M_t(\lambda_t,
J_{\lambda_t})$ (see~\eqref{EJt}). Moreover, every highest weight module
over $\mathfrak{sl}_2$ is either a Verma module or a finite-dimensional
simple module; this yields that $M_t(\lambda_t,J_{\lambda_t}) \cong
V_t(\lambda_t)$. Thus, $M(\lambda,J) \cong \otimes_{t=1}^T
V_t(\lambda_t)$ is simple. Moreover, condition $(M+)$ holds because
$\rho(h_\beta) = 1$ for all $\beta \in \Phi^+$, so that $\Psi_\lambda^+$
is empty.
\end{enumerate}
In all of these cases, the $x^\delta$-shifted character of $V(\lambda)$
is denormalized Lorentzian by Theorem~\ref{T5}. \qed
\end{example}

The next example was discussed in point~(3) above -- and in the
introduction. It identifies a large (within the
non-generic/non-antidominant) set of weights $\lambda$ for which the
simple module $V(\lambda)$ is infinite-dimensional and not Verma, and
whose characters we prove are shifted denormalized Lorentzian and hence
log-concave by Theorem~\ref{T5}. In particular, this goes beyond Huh et
al's results in \cite{HMMS} in ascertaining their Conjecture~12.

\begin{example}\label{Exgeneric}
For this example, let $\mathfrak{g} \neq \mathfrak{sl}_2(\mathbb{C})$ be
as above. Suppose $\lambda \in \mathfrak{h}^*$ is ``generic
non-antidominant'': specifically, say $(\lambda + \rho)(h_\alpha)$ is a
positive integer for a unique positive root $\alpha$, which is moreover
simple: $\alpha = \alpha_{i,i+1}^{(t)}$ for some $t \in [T]$ and $i \in
[n_t]$. Then $M(\lambda)$ contains the proper submodule
$f_\alpha^{\lambda(h_\alpha) + 1} m_\lambda$, hence is not simple; and
$\lambda \not\in \Lambda^+$ so $V(\lambda)$ is not finite-dimensional.

Nevertheless, one can prove Conjecture~\ref{Conj} for $V(\lambda)$.
Indeed, $J_\lambda = \{ (t,i) \}$, so $\Psi_\lambda^+$ is empty, and
hence Jantzen's criterion yields $M(\lambda, J_\lambda) = V(\lambda)$.
Now Theorem~\ref{T5} yields Conjecture~\ref{Conj}. \qed
\end{example}

The preceding construction generalizes to smaller dimensional
``exceptional'' subsets of $\mathfrak{h}^*$, as we now show.

\begin{example}\label{Exdim}
Suppose $\mathfrak{g} = \oplus_{t=1}^T \mathfrak{sl}_{n_t+1}(\mathbb{C})$
with $\sum_t n_t \geq 2$. Choose any nonempty proper subset of Dynkin
diagram nodes $\emptyset \subsetneq J \subsetneq \sqcup_{t=1}^T (t,
[n_t])$. Now choose a nonnegative integer $n_{t,i}$ for every node $(t,i)
\in J$; while for $(t,i) \not\in J$ we choose complex numbers $z_{t,i}$
such that no nonempty subset of them sums to an integer $\geq -\sum_t
n_t$. (E.g., this can be done by choosing $z_0 := 1$ and extending it to
a $\mathbb{Q}$-linearly independent set of $z_{t,i}$; or one can choose
distinct primes $p_{t,i}$ and let $z_{t,i} := 1/p_{t,i}$; or one can even
choose $z_{t,i}$ to be sufficiently negative integers.)

Finally, define $\lambda \in \mathfrak{h}^*$ via its action on the simple
coroots:
\[
\lambda(h_{t,i}) = \lambda(h_{\alpha_{i,i+1}^{(t)}}) = \begin{cases}
n_{t,i}, \qquad & \text{if } (t,i) \in J,\\
z_{t,i}, \qquad & \text{if } (t,i) \not\in J,
\end{cases}
\]
and extend by $\mathbb{C}$-linearity to all of $\mathfrak{h}$. It is
clear that
(a)~$J_\lambda = J$;
(b)~since $J$ is nonempty, $\lambda$ is not antidominant; and
(c)~since $J$ is a proper subset, $\lambda \not\in P^+$.
Thus -- even when $T=1$ -- we are beyond the cases covered
in~\cite{HMMS}. However, recall the paragraph after Definition~\ref{Dkpf}
discussing the graph $G_J$. Now by the choice of $z_{t,i}$,
and since $h_{\alpha_{ij}^{(t)}} = h_{\alpha_{i,i+1}^{(t)}} + \cdots +
h_{\alpha_{j-1,j}^{(t)}}$ for all $i<j$, $\Psi_\lambda^+$ is empty.
But then Theorem~\ref{Tjantzen} yields that $M(\lambda,J_\lambda)$ is
(neither a Verma nor finite-dimensional, but is) simple, i.e.,
$V(\lambda)$. Thus we again obtain Conjecture~\ref{Conj} from
Theorem~\ref{T5}. \qed
\end{example}

For our final example, we systematically analyze the simplicity of
parabolic Vermas for all highest weights in rank $2$.

\begin{example}
Suppose $\mathfrak{g} = \mathfrak{sl}_3(\mathbb{C})$, so that $J_\lambda
\subseteq \{ 1, 2 \}$ for all $\lambda \in \mathfrak{h}^*$. We classify
the highest weights $\lambda \in \mathfrak{h}^*$ for which $M(\lambda,
J_\lambda)$ is simple.
\begin{enumerate}
\item If $\lambda(h_1), \lambda(h_2) \in \mathbb{N}$ then
$M(\lambda,J_\lambda)$ is finite-dimensional and simple.

\item If $\lambda(h_1), \lambda(h_2) \in \mathbb{C} \setminus
\mathbb{N}$, then $J_\lambda = \emptyset$ and so $M(\lambda,J_\lambda) =
M(\lambda)$. Now there are two sub-cases.

First if $\lambda(h_1+h_2) + 1 \in \mathbb{C} \setminus \mathbb{N}$, then
$M(\lambda,J_\lambda)$ is again simple by Theorem~\ref{Thum}, since
$h_{\alpha_{13}} = h_1 + h_2$ and $\rho(h_1) = \rho(h_2) = 1$. In
particular, Theorem~\ref{T1} applies to $V(\lambda) = M(\lambda,
J_\lambda)$ and affirms Conjecture~\ref{Conj}. Note that if
$\lambda(h_1), \lambda(h_2) \in \mathbb{Z}$ then we recover (and affirm)
\cite[Conjecture~12]{HMMS}, but not otherwise, since $\lambda$ would not
be integral then.

Else $\Psi_\lambda^+ = \{ \beta := \alpha_{13} \}$. Hence condition
$(M+)$ fails, since $s_\beta(\gamma) \not\in \Phi_{J_\lambda} =
\emptyset$ for any $\gamma$. It follows by Theorem~\ref{Tjantzen} that
$M(\lambda,J_\lambda)$ is not simple.

\item Finally, say $\lambda(h_1) \in \mathbb{N} \not\ni \lambda(h_2)$
(the other case is similar by symmetry of the Dynkin diagram). Then
$J_\lambda = \{ 1 \}$ and $\Phi^+ \setminus \Phi^+_{J_\lambda} = \{
\alpha_{23}, \alpha_{13} \}$; in particular, $V(\lambda)$ is neither a
Verma nor finite-dimensional, so~\cite{HMMS} does not apply here. There
are again two sub-cases:

First if $\lambda(h_1 + h_2) + 1 \in \mathbb{C} \setminus \mathbb{N}$,
then $\lambda$ lies on a simple root affine-hyperplane and is ``generic''
as in Example~\ref{Exgeneric}. Hence $M(\lambda,J_\lambda)$ is simple and
Conjecture~\ref{Conj} holds, via Theorems~\ref{Tjantzen} and~\ref{T1}. As
above, if $\lambda(h_2) \in \mathbb{Z}$ then we recover (and affirm)
\cite[Conjecture~12]{HMMS}, while if $\lambda(h_2) \not\in \mathbb{Z}$
then \cite[Conjecture~12]{HMMS} does not apply.

Else $\Psi_\lambda^+ = \{ \beta := \alpha_{13} \}$. Now we claim that
$M(\lambda, J_\lambda)$ is not simple. Indeed, since $s_\beta(\gamma) \in
\Phi_{J_\lambda}$ we may assume that $s_\beta(\gamma) = \alpha_{12}$,
whence $\gamma = s_\beta(\alpha_{12}) = -\alpha_{23} \in \mathbb{Q}
\alpha_{12} + \mathbb{Q} \alpha_{13}$, as desired. Thus, condition $(M+)$
holds if and only if $(\lambda + \rho)(h_\gamma) = 0$. But
\[
(\lambda + \rho)(h_{-\alpha_{23}}) = -\lambda(h_2) - 1,
\]
which is not an integer (hence is nonzero) since $\lambda(h_2) \not\in
\mathbb{N}$. \qed
\end{enumerate}
Thus, the cases over $\mathfrak{sl}_3$ of \cite[Conjecture~12]{HMMS} that
were shown by Huh et al.\ to hold were when the integers $\lambda(h_1),
\lambda(h_2)$ are both nonnegative or both negative -- the first and
third ``quadrants'' of the lattice in the $XY$-plane, where $X,Y$ stand
for $\lambda(h_1), \lambda(h_2)$ respectively.
Our analysis affirms their conjecture over ``half'' of each of the other
two open quadrants -- more precisely, the lattice points in these, lying
on or below the line $X+Y=-2$.
After the above analysis, the only cases that remain unresolved are when
$\lambda(h_1) \in \mathbb{N}$ and $\lambda(h_2) \in [-1 - \lambda(h_1),
-1] \cap \mathbb{Z}$, or vice versa.
\end{example}
%}}}

\end{document}